\documentclass[preprint,12pt]{elsarticle}
\usepackage{amsmath,amsthm,amssymb}
\usepackage{epstopdf}
\usepackage{float}
\usepackage{makecell}
\usepackage{booktabs}
\usepackage{multirow}
\usepackage{diagbox}
\usepackage{mathdots}
\usepackage{caption}
\usepackage{subcaption}
\usepackage{extarrows}
\biboptions{numbers,sort&compress}
\usepackage[colorlinks,linkcolor=blue,anchorcolor=blue,citecolor=blue]{hyperref}
\usepackage{bm,bbm}

\newtheorem{lemma}{Lemma}
\newtheorem{theorem}{Theorem}
\newtheorem{proposition}{Proposition}
\newtheorem{corollary}{Corollary}

\newtheorem{remark}{\indent {Remark}}
\newtheorem{example}{Example}
\numberwithin{theorem}{section} 
\numberwithin{equation}{section}
\numberwithin{lemma}{section}
\numberwithin{example}{section}
\numberwithin{definition}{section}

\newcommand{\cred}[1]{{\color{black}  #1}}

\allowdisplaybreaks[4]

\journal{Journal}

\begin{document}
\begin{frontmatter}

\title{Multilevel Tau preconditioners for symmetrized multilevel Toeplitz systems with applications to solving space fractional diffusion equations}

\author[mymainaddress]{Congcong Li}\ead{22482245@life.hkbu.edu.hk}
\author[mymainaddress]{Sean Hon\corref{cor1}}\ead{seanyshon@hkbu.edu.hk}
\cortext[cor1]{Corresponding author.}

\address[mymainaddress]{Department of Mathematics, Hong Kong Baptist University, Kowloon Tong, Hong Kong SAR}

\begin{abstract}
 \cred{In this work, we develop a novel multilevel Tau matrix-based preconditioned method for a class of non-symmetric multilevel Toeplitz systems. This method not only accounts for but also improves upon an ideal preconditioner pioneered by [J. Pestana. Preconditioners for symmetrized Toeplitz and multilevel Toeplitz matrices. SIAM Journal on Matrix Analysis and Applications, 40(3):870-887, 2019]. The ideal preconditioning approach was primarily examined numerically in that study, and an effective implementation was not included. To address these issues, we first rigorously show in this study that this ideal preconditioner can indeed achieve optimal convergence when employing the minimal residual (MINRES) method, with a convergence rate is that independent of the mesh size. Then, building on this preconditioner, we develop a practical and optimal preconditioned MINRES method. To further illustrate its applicability and develop a fast implementation strategy, we consider solving Riemann-Liouville fractional diffusion equations as an application. Specifically, following standard discretization on the equation, the resultant linear system is a non-symmetric multilevel Toeplitz system, affirming the applicability of our preconditioning method. Through a simple symmetrization strategy, we transform the original linear system into a symmetric multilevel Hankel system. Subsequently, we propose a symmetric positive definite multilevel Tau preconditioner for the symmetrized system, which can be efficiently implemented using discrete sine transforms. Theoretically, we demonstrate that mesh-independent convergence can be achieved. In particular, we prove that the eigenvalues of the preconditioned matrix are bounded within disjoint intervals containing $\pm 1$, without any outliers. Numerical examples are provided to critically discuss the results, showcase the spectral distribution, and support the efficacy of our preconditioning strategy.}
\end{abstract}

\end{frontmatter}

\section{Introduction}
In recent years, there has been growing interest in developing effective preconditioners for symmetrized Toeplitz systems. The underlying idea stems from \cite{doi:10.1137/140974213}, in which the original (real) nonsymmetric Toeplitz matrix $T_m \in \mathbb{R}^{m \times m}$ is symmetrized by premultiplying it with the anti-identity matrix
\begin{equation}
Y_m = {\begin{bmatrix}{}
& & & 1  \\
& & 1 &   \\
& \iddots & &\\
1 & & &
\end{bmatrix}} \in \mathbb{R}^{m \times m},
\end{equation} i.e., $[Y_m]_{j,k}=1$ if and only if $j+k=m+1$ and $[Y_m]_{j,k}=0$ elsewhere. For the now symmetric system with $Y_m T_m$, the minimal residual (MINRES) method can be used. Its convergence behaviour is related only to the eigenvalues and therefore effective preconditioners can be constructed exploiting the spectral information. 

While the generalized minimal residual (GMRES) method is a common preconditioned solver applied to the original nonsymmetric Toeplitz system with $T_m$, it has a significant drawback: its preconditioning is largely heuristic, as discussed in \cite{ANU:9672992}. Besides, in general, the convergence behaviors of GMRES cannot be rigorously analyzed solely through the knowledge of the eigenvalues, according to \cite{Greenbaum_1996}.

For the symmetrized Toeplitz matrix $Y_mT_m$, various preconditioning techniques have been proposed. Absolute value circulant preconditioners were first proposed in \cite{doi:10.1137/140974213}, with the preconditioned matrix eigenvalues shown to cluster around $\pm 1$ under certain conditions. Additionally, Toeplitz and band-Toeplitz preconditioners have been proposed in \cite{Pestana2019} and \cite{Hon_SC_Wathen}, respectively.

The same symmetrization preconditioning approach has also been employed in solving evolutionary differential equations. Initially introduced for solving ordinary and partial differential equations (PDEs) in \cite{McDonald2017} and \cite{doi:10.1137/16M1062016}, respectively, this approach has since been extended to a number of PDE problems, as discussed in \cite{HonFungDongSC_2023} and \cite{hondongSC2023}, for example.

\cred{Notably, this symmetrization approach was further generalized to the multilevel Toeplitz case in \cite{Pestana2019}. The author proposed a multilevel Toeplitz preconditioner as an ideal preconditioner for a class of symmetrized multilevel Toeplitz systems. Although this approach was shown to be effective in some cases, as demonstrated numerically in the same work, its overall preconditioning effect and the issues related to fast implementation were not fully explored.

Motivated by the need to address these issues, in this work, we first rigorously demonstrate that this ideal preconditioner can indeed achieve optimal convergence when employing the MINRES method, with a convergence rate that is independent of the mesh size. Then, to further illustrate its applicability and develop a fast implementation strategy, we consider solving Riemann-Liouville (R.-L.) fractional diffusion equations as an application. The findings demonstrate that Tau preconditioning constitutes an optimal strategy for symmetrized multilevel Toeplitz systems that emerge from the discretization of fractional diffusion equations.

It should be noted that, although some commonly used preconditioners, {such as the aforementioned circulant and Toeplitz preconditioners}, have been proposed for symmetrized multilevel Toeplitz systems, the performance of Tau preconditioners remains unexplored in this context. This is noteworthy considering their demonstrated effectiveness in symmetric Toeplitz systems, as studied in several works (see, for example, \cite{D1,DS,Slinear}).

As a model application, we consider the following space R.-L. fractional diffusion equation}
\begin{equation}\label{eq:fde}
\left\{ \begin{array}{ll}
    \frac{\partial u(x,t)}{\partial t}-\sum_{i=1}^{d}\Big( d_{i,+}\frac{\partial_+^{\alpha_i}}{\partial x^{\alpha_i}} + d_{i,-}\frac{\partial_-^{\alpha_i}}{\partial x^{\alpha_i}}\Big) u(x{,t}) = f(x{,t}), &x\in\Omega, {t\in (0,T],}\\
    u(x{,t}) = 0, &x \in \partial\Omega, \\    {u(x,0) = u_0(x),} &{x\in\Omega,}
    \end{array} \right.\,
\end{equation}
where $\Omega=\prod_{i=1}^{d}(a_i,b_i)$ is an open hyper-rectangle, $\partial\Omega$ denotes the boundary of $\Omega$, $\alpha_i \in (1,2)$ are the fractional derivative orders, $f(x{, t})$ is the source term, $u_0(x)$ is a given function, and the diffusion coefficients $d_{i,\pm}$ are nonnegative constants. The left- and the right-sided R.-L. fractional derivatives in \eqref{eq:fde} are defined by
\begin{equation*}
\begin{split}
\frac{\partial_+^{\alpha_i} u(x{,t})}{\partial x^{\alpha_i}}&=\frac{1}{\Gamma(2-{\alpha_i})}\frac{\partial^{2}}{\partial x_{i}^{2}}\int_{a_i}^{x_i}\frac{u(x_{1},x_{2},\dots,x_{i-1},\xi,x_{i+1},\dots,x_{d}{,t})}{(x_{i}-\xi)^{\alpha_{i}-1}}{\rm d}\xi, \\
\frac{\partial_-^{\alpha_i} u(x{,t})}{\partial x^{\alpha_i}}&=\frac{1}{\Gamma(2-{\alpha_i})}\frac{\partial^{2}}{\partial x_{i}^{2}}\int_{x_i}^{b_i}\frac{u(x_{1},x_{2},\dots,x_{i-1},\xi,x_{i+1},\dots,x_{d}{,t})}{(\xi-x_i)^{\alpha_i-1}}{\rm d}\xi,
\label{right-LR}
\end{split}
\end{equation*}
respectively, where ${\Gamma(\cdot)}$ denotes the gamma function.

Denote the set of all positive integers and the set of all nonnegative integers by $\mathbb{N}^{+}$ and $\mathbb{N}$, respectively. For any $m, k \in \mathbb{N}$ with $m \leq k$, define the set $m \wedge k:=$ $\{m, m+1, \ldots, k-1, k\}$. For $n_i \in \mathbb{N}^{+}, i \in 1 \wedge d$, denote by $h_i=\left(b_i-a_i\right) /\left(n_i+1\right)$, the stepsize along the $i$-th direction. Denote
	$$
	n=\prod_{i=1}^d n_i, \quad n_1^{-}=n_d^{+}=1, \quad n_i^{-}=\prod_{j=1}^{i-1} n_j,~i \in 2 \wedge d, \quad n_k^{+}=\prod_{j=k+1}^d n_j,~ k \in 1 \wedge(d-1) .
	$$

 \cred{The Crank–Nicolson} method is applied for the temporal derivative, and the \cred{weighted and shifted} Gr\"unwald scheme \cred{\cite{Tian2015,vong2019second}} is used for the space fractional derivatives. \cred{The time interval $[0,T]$ is partitioned into $M$ sub-intervals, each with a time step size of $\tau := T/M$.} This results in a scheme that is \cred{second-order} accurate in both time and space, leading to the linear system
\begin{equation}\label{eqn:Au_second}
  		\cred{\underbrace{(\nu I_n + B_n)}_{=:A_n} \underbrace{ \mathbf{u}^{k+1} }_{ =:\mathbf{x}} = \underbrace{(\nu I_n - B_n) \mathbf{u}^{k}  + \mathbf{f}^{k+\frac{1}{2}}}_{=:\mathbf{b}},}
\end{equation}
	where $\nu=\frac{1}{\tau}$ and $I_m$ represents an $m \times m$ identity matrix. The vector \cred{$\mathbf{f}^{k+\frac{1}{2}}$} is known from the numerical scheme used and denotes the sampling of $f$ at the grid points. The matrix $A_{n}$ is a nonsymmetric multilevel Toeplitz matrix. More precisely, the coefficient matrix $A_{n}$ can be expressed as
 \cred{
	\begin{eqnarray}\label{eqn:main_matrix_second}
	A_n &=& \nu I_n + B_n = \nu I_n  + \sum_{i=1}^d\left(v_{i,+} W_i+v_{i,-} W_i^{\top}\right),\\\nonumber
 &&\quad W_i= I_{n_i^{-}} \otimes L_{n_i}^{\left(\alpha_i\right)} \otimes {I}_{n_i^{+}}, \quad v_{i,+}:=\frac{d_{i,+}}{2h_i^{\alpha_i}}, \quad v_{i,-}:=\frac{d_{i,-}}{2h_i^{\alpha_i}},
	\end{eqnarray}
}in which  ` $\otimes$ ' denotes the Kronecker product. The matrix $L_{n_i}^{(\alpha_i)}$ is defined as
	\begin{equation}\label{eq:grunwaldmatrix_second}
\cred{
L_{n_i}^{(\alpha_i)}=
-\begin{bmatrix}
w_1^{(\alpha_i)} & w_0^{(\alpha_i)} & 0 & \cdots & 0 \\
			w_2^{(\alpha_i)} & w_1^{(\alpha_i)} & w_0^{(\alpha_i)} & \ddots & \vdots \\
			\vdots & w_2^{(\alpha_i)} & w_1^{(\alpha_i)} & \ddots & 0 \\
			\vdots & \ddots & \ddots & \ddots & w_0^{(\alpha_i)} \\
			w_{n_i}^{(\alpha_i)} & \cdots & \cdots & w_2^{(\alpha_i)} & w_1^{(\alpha_i)}\\
\end{bmatrix} \in \mathbb{R}^{n_i \times n_i},
}
\end{equation} \cred{where the coefficients $w_0^{(\alpha_i)}=\frac{\alpha_i}{2}g_0^{(\alpha_i)}, w_k^{(\alpha_i)}=\frac{\alpha_i}{2}g_k^{(\alpha_i)} + \frac{2-\alpha_i}{2}g_{k-1}^{(\alpha_i)}$ for $k \geq 1$. Note that the coefficients $g_k^{(\alpha_i)}$ are given by
\begin{equation}
    g_0^{(\alpha_i)}=1, \quad g_k^{(\alpha_i)}=\left( 1 -\frac{\alpha_i +1}{ k} \right) g_{k-1}^{(\alpha_i)}
\end{equation}
for $k \geq 1$.}

From \cite{vong2019second} and following the notation to be introduced in Section \ref{sec:prelim}, we know that $L_{n_i}^{(\alpha_i)}=T_{n_i}[g_{\alpha_i}]$ is a Toeplitz matrix generated by the complex-valued function

\cred{
\begin{align}\label{eqn:gen_function_second}
    g_{\alpha_i}(\theta)&=-\sum_{k=-1}^{\infty} w_{k+1}^{({\alpha_i})} \mathrm{e}^{\mathrm{i} k \theta}=-\left[ \frac{\alpha_i}{2}\mathrm{e}^{-\mathbf{i} \theta} \left(1-\mathrm{e}^{\mathbf{i} \theta}\right)^{\alpha_i}+ \frac{2-\alpha_i}{2}  \left(1-\mathrm{e}^{\mathbf{i} \theta}\right)^{\alpha_i}  \right].
\end{align}
}
Thus, the generating function of $A_n$ is associated by
\begin{equation}\label{eqn:main_gen_function_second}
    f_{\boldsymbol{\alpha}}(\boldsymbol{\theta}) = \nu + \sum_{i=1}^d  v_{i,+}g_{\alpha_i}(\theta_i) + v_{i,-}g_{\alpha_i}(-\theta_i),
\end{equation}
where $\boldsymbol{\alpha} = (\alpha_1, \dots, \alpha_d)$.

\cred{ 
Alternatively, if a backward Euler method is applied for the temporal derivative, and the shifted Gr\"unwald scheme \cite{Meerschaert2004,Meerschaert2006} is used for the space fractional derivatives. This result in a scheme that is first-order accurate in both time and space, leading to the linear system
\begin{equation}\label{eqn:Au_first}
		\widetilde{A}_n \underbrace{\mathbf{u}^k}_{=:{\mathbf{x}}} = \underbrace{\nu \mathbf{u}^{k-1}  + \mathbf{f}^k}_{=:{\mathbf{b}}}.
\end{equation}
	In this case, the nonsymmetric multilevel Toeplitz matrix $\widetilde{A}_{n}$ can be expressed as
	\begin{eqnarray}\label{eqn:main_matrix_first}
	\widetilde{A}_n = \nu I_n+ \sum_{i=1}^d\left( \widetilde{v}_{i,+} \widetilde{W}_i+ \widetilde{v}_{i,-} \widetilde{W}_i^{\top}\right), \quad \widetilde{W}_i= I_{n_i^{-}} \otimes \widetilde{L}_{n_i}^{\left(\alpha_i\right)} \otimes {I}_{n_i^{+}},  \\ \nonumber
 \widetilde{v}_{i,+}:=\frac{d_{i,+}}{h_i^{\alpha_i}}, \quad \widetilde{v}_{i,-}:=\frac{d_{i,-}}{h_i^{\alpha_i}}.
	\end{eqnarray}
 The matrix $\widetilde{L}_{n_i}^{(\alpha_i)}$ is defined as
	\begin{equation}\label{eq:grunwaldmatrix_first}
\widetilde{L}_{n_i}^{(\alpha_i)}=
-\begin{bmatrix}
\widetilde{g}_1^{(\alpha_i)} & \widetilde{g}_0^{(\alpha_i)} & 0 & \cdots & 0 \\
			\widetilde{g}_2^{(\alpha_i)} & \widetilde{g}_1^{(\alpha_i)} & \widetilde{g}_0^{(\alpha_i)} & \ddots & \vdots \\
			\vdots & \widetilde{g}_2^{(\alpha_i)} & \widetilde{g}_1^{(\alpha_i)} & \ddots & 0 \\
			\vdots & \ddots & \ddots & \ddots & \widetilde{g}_0^{(\alpha_i)} \\
			\widetilde{g}_m^{(\alpha_i)} & \cdots & \cdots & \widetilde{g}_2^{(\alpha_i)} & \widetilde{g}_1^{(\alpha_i)}\\
\end{bmatrix} \in \mathbb{R}^{n_i \times n_i},
\end{equation} where the coefficients $\widetilde{g}_k^{(\alpha_i)}$ given by $\widetilde{g}_k^{(\alpha_i)}=(-1)^k\binom{\alpha_i}{k}=\frac{(-1)^k}{k!}\alpha_i(\alpha_i-1)\cdots(\alpha_i-k+1), ~k\geq 0,$ where $\binom{\alpha_i}{0}=1$. 

From \cite{donatelli161}, we know that $\widetilde{L}_{n_i}^{(\alpha_i)}=T_{n_i}[\widetilde{g}_{\alpha_i}]$ is a Toeplitz matrix generated by the complex-valued function $$\widetilde{g}_{\alpha_i}(\theta)=-e^{-\mathbf{i}\theta}\left(1-e^{\mathbf{i}\theta}\right)^{\alpha_i}.$$
Thus, the generating function of $\widetilde{A}_n$ is associated by
{
\begin{equation}\label{eqn:main_gen_function_first}
    \widetilde{f}_{\boldsymbol{\alpha}}(\boldsymbol{\theta}) = \nu + \sum_{i=1}^d   \widetilde{v}_{i,+}\widetilde{g}_{\alpha_i}(\theta_i) +  \widetilde{v}_{i,-}\widetilde{g}_{\alpha_i}(-\theta_i).
\end{equation}}
}
Instead of solving \cred{\eqref{eqn:Au_second} (or \eqref{eqn:Au_first})} directly, we employ the MINRES method to solve the following equivalent system denoted by
\cred{
\begin{equation}\label{eqn:preconditioned}
    Y_n A_n \mathbf{x}=Y_n \mathbf{b} \quad (\textrm{or} \quad Y_n \widetilde{A}_n {\mathbf{x}}=Y_n {\mathbf{b}}).
\end{equation}
}
In an endeavor to expedite the convergence of the MINRES algorithm applied to the symmetrized system $Y_n A_n \mathbf{x} = Y_n \mathbf{b}$ \cred{(or $Y_n \widetilde{A}_n {\mathbf{x}}=Y_n {\mathbf{b}})$}, the absolute value circulant preconditioner proposed in \cite{doi:10.1137/140974213} was considered as a potential solution. However, as demonstrated in the numerical experiments of \cite{Pestana2019}, its effectiveness has been shown to be unsatisfactory. Furthermore, a comprehensive analysis was provided in \cite{Hon_SC_Wathen} showing that circulant preconditioners do not generally ensure rapid convergence. This is particularly true for ill-conditioned Toeplitz systems, where the preconditioned matrix may exhibit eigenvalue outliers that are very close to zero. Moreover, circulant preconditioners are known to be suboptimal in the context of multilevel Toeplitz systems, as discussed in \cite{SeTy-nega1, SeTy-nega2}. As an alternative, \cite{Hon_SC_Wathen} proposed a band-Toeplitz preconditioner for symmetrized Toeplitz systems, specifically when the generating functions exhibit zeros of even order. Yet, this approach is not applicable for the matrix $Y_n A_n$ considered in this work. This is because the function $f_{\boldsymbol{\alpha}}$ associated with $Y_n A_n$ has a zero of fractional order $\boldsymbol{\alpha}$ as discussed in \cite{donatelli161}, a scenario where the band-Toeplitz preconditioner does not apply.

Recently, the use of the symmetric part of the multilevel Toeplitz coefficient matrix as a preconditioner for the target nonsymmetric multilevel Toeplitz systems was proposed in \cite{Pestana2019}. While this ideal \cred{symmetric part} preconditioner can facilitate fast MINRES convergence, the computationally expensive nature of its implementation limits practical application. Moreover, the effectiveness of the preconditioner has been observed numerically but has not yet been theoretically proven.

In this context, the main contributions of our work are twofold:
\begin{enumerate}
    \item We show that the ideal preconditioner proposed in \cite{Pestana2019} can indeed lead to a MINRES convergence rate that is independent of the mesh size for \cred{a class of symmetrized multilevel Toeplitz systems. Moreover, building on this preconditioner, we develop a practical and optimal preconditioned MINRES method for these symmetrized systems (see Corollary \ref{cor:practical_MINRES_convergence}).}
    
    \item \cred{To illustrate our preconditioning approach, we propose a multilevel Tau preconditioner as an practical preconditioner} for $Y_n A_n \mathbf{x} = Y_n \mathbf{b}$ \cred{(or $Y_n \widetilde{A}_n {\mathbf{x}}=Y_n {\mathbf{b}})$}. This preconditioner balances the effectiveness of preconditioning with computational feasibility. Our approach offers a practical strategy for preconditioning fractional diffusion equations and achieves mesh-independent MINRES convergence. To the best of our knowledge, this study is the first to show that Tau preconditioning is an optimal choice for \cred{a range of} symmetrized multilevel Toeplitz systems stemming from space R.-L. fractional diffusion equations, \cred{where both first and second order schemes are considered.}
\end{enumerate}

The paper is organized as follows: Our proposed preconditioner is defined in Section \ref{sec:Tau}. Section \ref{sec:prelim} reviews some preliminary results on multilevel Toeplitz matrices. Section \ref{sec:main} presents our main results on the effectiveness of our proposed preconditioners. Numerical examples in Section \ref{sec:numerical} demonstrate the expected performance of our proposed preconditioners, and support the theoretical results concerning the eigenvalues of the associated preconditioned matrices.

\section{\cred{\textbf{{Proposed preconditioners}}}}\label{sec:Tau}
	In this section, our multilevel Tau preconditioner $P_n$ for \eqref{eqn:preconditioned} is presented.
	
	For a symmetric Toeplitz matrix $T_m\in\mathbb{R}^{m\times m}$ with $(t_1,t_2,...,t_m)^{\top}\in\mathbb{R}^{m}$, define its $\tau$-matrix \cite{Bini1990} approximation as
	\begin{equation}\label{tauopdef}
		\tau(T_m):=T_m-H_m,
	\end{equation}
	where $H_m$ is the Hankel matrix with $(t_3,t_4,...,t_m,0,0)^{\top}$ as its first column and $(0,0,t_m,...,t_4,t_3)^{\top}$ as its last column. A crucial property of the Tau matrix defined in \eqref{tauopdef}  is that it is diagonalizable by sine transform matrix, i.e.,
	\begin{equation}\label{taumatdiag}
		\tau(T_m)=S_m Q_m S_m,
	\end{equation}
	where $Q_m=[{\rm diag}(q_i)]_{i=1}^{m}$ is a diagonal matrix with
	\begin{equation}\label{sigmicomp}
		q_{i}=t_1+2\sum\limits_{j=2}^{m}t_j\cos\left(\frac{\pi i(j-1)}{m+1}\right),\quad  i\in 1\wedge m.
	\end{equation} 
	\begin{equation*}
		S_m:= \left[\sqrt{\frac{2}{m+1}}\sin\left(\frac{\pi jk}{m+1}\right)\right]_{j,k=1}^{m}
	\end{equation*}
	is a sine transform matrix. It is easy to verify that $S_m$ is a symmetric orthogonal matrix, i.e., $S_m=S_m^{\top}=S_m^{-1}$. The product of matrix $S_m$ and a given vector of length $m$ can be fast computed in $\mathcal{O}(m\log m)$ operations using discrete sine transform (DSTs) \cite{BC83}. Let ${\bf e}_{m,i}\in\mathbb{R}^{m}$ denotes the $i$-th column of the $m\times m$ identity matrix.
	We also note that the $m$ numbers $\{q_i\}_{i=1}^{m}$ defined in \eqref{sigmicomp} can be computed by
	\begin{equation*}
		(q_1, q_2, \cdots, q_m)^{\top}={\rm diag}(S_m{\bf e}_{m,1})^{-1}[S_m\tau(T_m){\bf e}_{m,1}].
	\end{equation*}
	From the equation above, we know that the computation of $\{q_i\}_{i=1}^{m}$ requires only $\mathcal{O}(m\log m)$ operations.
	
	For a real square matrix $Z$, denote the symmetric part of $Z$ as
	$$
	\mathcal{H}(Z):=\frac{Z+Z^{\top}}{2}.
	$$
	Then, our multilevel Tau preconditioner $P_n$ for $Y_nA_n$ in \eqref{eqn:preconditioned} is defined as follows
	\begin{equation}\label{eqn:main_krecondition_second}
		P_n:= \nu I_{n} + \sum_{i=1}^d\left(v_{i,+}+v_{i,-}\right) I_{n_i^{-}} \otimes \tau\left(	\mathcal{H}\left(L_{n_i}^{\left(\alpha_i\right)} \right)\right) \otimes I_{n_i^{+}}.
		\end{equation}
	\cred{
 Similarly, in the first-order case, a multilevel Tau preconditioner $\widetilde{P}_n$ for $Y_n\widetilde{A}_n$ in \eqref{eqn:preconditioned} can be defined as follows
{	\begin{equation}\label{eqn:main_krecondition_first}
		\widetilde{P}_n:= \nu I_{n} + \sum_{i=1}^d\left( \widetilde{v}_{i,+}+ \widetilde{v}_{i,-}\right) I_{n_i^{-}} \otimes \tau\left(	\mathcal{H}\left(\widetilde{L}_{n_i}^{\left(\alpha_i\right)} \right)\right) \otimes I_{n_i^{+}}.
		\end{equation}
}
 \begin{remark}\label{remark:first_order_precon}
We remark that in the steady case, $\widetilde{P}_n$ coincides with the preconditioner proposed in \cite{Lin_etc_2023}, with the only difference being the addition of the matrix $\nu I_n$. A related multilevel Tau preconditioner was also mentioned in \cite{Huang_2022} for Riesz fractional diffusion equations using a first-order scheme. However, our approach and convergence analysis differ significantly as we employ the MINRES method, whereas the cited methods utilize GMRES. Moreover, regarding implementation, the preconditioner proposed in \cite{Lin_etc_2023} is used in a two-sided manner, where GMRES operates on $\widetilde{P}_n^{-1/2} \widetilde{A}_n\widetilde{P}_n^{-1/2}\tilde{{\mathbf{x}}}=\widetilde{P}_n^{-1/2}{\mathbf{b}}$, where $\tilde{{\mathbf{x}}}=\widetilde{P}_n^{1/2} \mathbf{x}$. In contrast, our implementation is one-sided, with MINRES applied to $\widetilde{P}_n^{-1} Y_n \widetilde{A}_n {\mathbf{x}} =\widetilde{P}_n^{-1}Y_n{\mathbf{b}}$, as is typically done.
\end{remark}
 
 }

	From \eqref{taumatdiag} \& \eqref{sigmicomp} and properties of the one-dimensional sine transform matrix $S_m$, we know that $P_n$ is diagonalizable by a $d$-dimension sine transform matrix, namely,
	\begin{align*}
			P_n=S \Lambda S,\label{btaudiag} \quad S:=S_{n_1}\otimes \dots \otimes S_{n_d},\quad \Lambda:= \nu I_{n} + \sum\limits_{i=1}^{d}\left(v_{i,+}+v_{i,-}\right)I_{n_i^{-}}\otimes \Lambda_i\otimes I_{n_i^{+}},
	\end{align*}
 where $\Lambda_i$ contains the eigenvalues of $\tau\left(	\mathcal{H}\left(L_{n_i}^{\left(\alpha_i\right)} \right)\right)$. Thus, the product of $P_n$ and a given vector can be efficiently computed in $\mathcal{O}(n\log n)$ operations using DSTs.
	
 In Section \ref{sec:main}, we will show that the eigenvalues of the preconditioned matrix \( P_n^{-1}Y_n A_n \) \cred{(or \( \widetilde{P}_n^{-1}Y_n \widetilde{A}_n \))} are contained in a disjoint interval enclosing \(\pm 1\), which leads to theoretically guaranteed matrix size/mesh-independent convergence when MINRES is applied.

\section{Preliminaries on multilevel Toeplitz matrices}
\label{sec:prelim}

In this section, we provide some useful background knowledge regarding multilevel Toeplitz and Hankel matrices.

Now consider the Banach space $L^1([-\pi,\pi]^k)$ of all complex-valued Lebesgue integrable functions over $[-\pi,\pi]^k$, equipped with the norm
\[
\|f\|_{L^1} = \frac{1}{(2\pi)^k}\int_{[-\pi,\pi]^k} |f({\boldsymbol{\theta}})|\,{\rm d} {\theta} < \infty,
\]
where ${\rm d} {\theta}={\rm d} {\theta_1}\cdots{\rm d} {\theta_k}$ denotes the volume element with respect to the $k$-dimensional Lebesgue measure.

Let $f:$~$[-\pi,\pi]^k\to \mathbb{C}$ be a function belonging to $L^1([-\pi,\pi]^k)$ and periodically extended to $\mathbb{R}^k$. The multilevel Toeplitz matrix $T_{{n}}[f] $ of size $n\times n$ with $n= n_1n_2 \dots n_k$ is defined as
\begin{equation*}
T_{{n}}[f] =\sum_{|j_1|<n_1}\ldots \sum_{|j_k|<n_k} J_{n_1}^{j_1} \otimes \cdots\otimes J_{n_k}^{j_k} a_{(\mathbf{j})}, \qquad \mathbf{j}=(j_1,j_2,\dots,j_k)\in \mathbb{Z}^k,
\end{equation*}
where
\[ a_{(\mathbf{j})}=a_{(j_1,\dots, j_k)}=\frac{1}{(2\pi)^k}\int_{[-\pi,\pi]^k}f({\boldsymbol{\theta}}){\rm e}^{\mathbf{i}\left\langle { \bf j},{\boldsymbol{\theta}}\right\rangle}\, {\rm d}{\theta},
\quad \left\langle { \bf j},{\boldsymbol{\theta}}\right\rangle=\sum_{t=1}^kj_t\theta_t, \quad \mathbf{i}^2=-1, \]
are the Fourier coefficients of $f$ and $J^{j}_{m}$ is the $m \times m$ matrix whose $(l,h)$-th entry equals 1 if $(l-h)=j$ and $0$ otherwise. The function $f$ is called the generating function of $T_{{n}}[f]$.

It is easy to prove that (see e.g., \cite{MR2108963,MR2376196,Chan:1996:CGM:240441.240445,GaroniCapizzano_two}) if $f$ is real-valued, then $T_{{n}}[f]$ is Hermitian; if $f$ is real-valued and nonnegative, but not identically zero almost everywhere, then $T_{{n}}[f]$ is Hermitian positive definite; if $f$ is real-valued and even, $T_{{n}}[f]$ is (real) symmetric.

Throughout this work, we assume that $f\in L^1([-\pi,\pi]^k)$ and is periodically extended to $\mathbb{R}^k$.

Similar to a multilevel Toeplitz matrix, we can define a multilevel Hankel matrix as 
\begin{equation*}
H_n [f] = \sum_{j_1 = 1}^{2n_1-1}\dots\sum_{j_k = 1}^{2n_k-1}K_{n_1}^{(j_1)}\otimes \dots \otimes K_{n_k}^{(j_k)} {a}_{(j_1,\dotsc,j_k)},
\end{equation*}
where $K^{(k)}_r \in \mathbb{R}^{r\times r}$ is the matrix whose $(i,j)$-th entry is one if $i+j=k+1$ and is zero otherwise. Clearly, a multilevel Hankel matrix is symmetric. 

Multilevel Toeplitz matrices can be symmetrized by the permutation matrix $Y_n \in \mathbb{R}^{n \times n}$, namely, $Y_{n} = Y_{n_1}\otimes \dots \otimes Y_{n_k}$. Knowing that 
$Y_r J_r^{(k)} = K_r^{(r-k)}$, we can easily show that 
 \begin{align*}
Y_n T_n[f]
=& \sum_{j_1 = -n_1+1}^{n_1-1}\dots\sum_{j_k = -n_k+1}^{n_k-1}\left((Y_{n_1}J_{n_1}^{(j_1)})\otimes \dots \otimes (Y_{n_k}J_{n_k}^{(j_k)})\right) {a}_{(j_1,\dotsc,j_k)}\\
=& \sum_{j_1 = -n_1+1}^{n_1-1}\dots\sum_{j_k = -n_k+1}^{n_k-1}\left( K_{n_1}^{(n_1-j_1)}\otimes \dots \otimes K_{n_k}^{(n_k-j_k)}\right){a}_{(j_1,\dotsc,j_k)})\\
=& \sum_{j_1 = 1}^{2n_1-1}\dots\sum_{j_k = 1}^{2n_k-1}K_{n_1}^{(j_1)}\otimes \dots \otimes K_{n_k}^{(j_k)} {{b}}_{(j_1,\dotsc,j_k)},
\end{align*}
where  $b_{(j_1,\dotsc,j_k)} = a_{(n_1-j_1,\dotsc,n_k-j_k)}$. Hence, $Y_n T_n[f]$  is a symmetric multilevel Hankel matrix. For more properties regarding multilevel Hankel matrices, see \cite{FaTi00_2000}.

A crucial aspect of developing effective preconditioners for $Y_nT_{n}[f]$ is understanding its asymptotic spectral distribution associated with $f$. This was established for the uni-level case in \cite{MazzaPestana2018, Ferrari2019, Hon_M_SC_2019} and later generalized to the multilevel case in \cite{MazzaPestana2021, Ferrari2021}.

\section{Main results}\label{sec:main}

\cred{

The main results of this work are divided into the following subsections.

\subsection{Convergence analysis of the ideal preconditioner for symmetrized multilevel Toeplitz systems with $Y_n T_{n}$ }

In this subsection, we provide a result as a straightforward consequence of the following lemma, which show that the ideal preconditioner $\mathcal{H}(  T_{n})$ can achieve optimal convergence for a class of symmetrized multilevel Toeplitz systems with $Y_n T_{n}$.

\begin{lemma}\cite[Theorem 4.2]{Pestana2019}\label{lemma:pestana} 
Let $f \in L^1([-\pi, \pi]^d)$ and let $f=\mathrm{Re}(f)+ \mathbf{i} \mathrm{Im}(f)$, where $\mathrm{Re}(f)$ and $\mathrm{Im}(f)$ are real-valued functions with $\mathrm{Re}(f)$ essentially positive. Additionally, let $T_{{n}}[f] \in \mathbb{R}^{n\times n}$ be the multilevel Toeplitz matrix generated by $f$ and let $\mathcal{H}( [f] )=(T_{{n}}[f]+T_{{n}}[f]^{\top} ) / 2$. Then, the eigenvalues of $\mathcal{H}( [f] )^{-1}  Y_{n}T_{{n}}[f] $ lie in $[-1-\epsilon, -1] \cup [1, 1+\epsilon]$, where
$
\epsilon<\operatorname{essup}_{\boldsymbol{\theta}  \in [-\pi, \pi]^d}\left|\frac{\mathrm{Im}(f)(\boldsymbol{\theta}) }{\mathrm{Re}(f)(\boldsymbol{\theta}) }\right|.
$
\end{lemma}

Since all eigenvalues of $\mathcal{H}( [f] )^{-1}Y_{n}T_{n}[f]$ are contained within the disjoint intervals $[-\hat{\beta},-\check{\beta}] \cup [\check{\beta}, \hat{\beta}]$ with no outliers, optimal convergence can be achieved according to a well-known classical result on the convergence of MINRES (see, for example, \cite{ElmanSilvesterWathen2004}). Namely, the $k$-th residual of $\mathbf{r}^{(k)}$ satisfies  
\begin{equation}\label{eqn:MINRES_bound}
\frac{\|\mathbf{r}^{(k)}\|_{2}}{\|\mathbf{r}^{(0)}\|_{2}}\leq 2\left(\frac{\hat{\beta}/\check{\beta} -1}{\hat{\beta}/\check{\beta} +1}\right)^{[k/2]}.
\end{equation} By letting $\check{\beta} = 1$ and $\hat{\beta} = 1+\epsilon$, the following corollary immediately follows:

\begin{corollary}\label{cor:ideal_MINRES_convergence}
Let $f \in L^1([-\pi, \pi]^d)$ satisfy the assumptions made in Lemma \ref{lemma:pestana}. Then, the preconditioned MINRES method for the system $\mathcal{H}( [f] )^{-1}Y_n T_{{n}}[f]$ has a convergence rate independent of $\mathbf{n}$, i.e., the residuals generated by the MINRES method satisfy
$
\frac{\|\mathbf{r}^{(k)}\|_{2}}{\|\mathbf{r}_n^{(0)}\|_{2}}\leq 2\omega^{[k]},
$
where $\mathbf{r}^{(k)} =  \mathcal{H}( [f] )^{-1}Y_{n} \mathbf{b} - \mathcal{H}( [f] )^{-1}  Y_{n} T_{{n}}[f]\mathbf{\tilde{u}_* }^{(k)}$, $\mathbf{\tilde{u}_* }^{(k)}$ denotes the $k$-th iterate by MINRES, $\mathbf{\tilde{u}_* }^{(0)}$ denotes an arbitrary initial guess, and $\omega$ is a constant independent of $\mathbf{n}$ defined as follows
$$
\omega := \sqrt{\frac{\epsilon }{2 + \epsilon}} \in (0,1),
$$
with $\epsilon<\operatorname{essup}_{\boldsymbol{\theta}  \in [-\pi, \pi]^d}\left|\frac{\mathrm{Im}(f)(\boldsymbol{\theta}) }{\mathrm{Re}(f)(\boldsymbol{\theta}) }\right|$.
\end{corollary}

With modifications, Corollary \ref{cor:ideal_MINRES_convergence} can turn into a practical preconditioned MINRES method. Now, suppose there is a symmetric positive definite preconditioner $\mathcal{P}$ that can be implemented efficiently. Also, it is assumed to be spectrally equivalent to $\mathcal{H}( [f] )$, in the sense that there exist two positive numbers $\check{c}$ and $\hat{c}$ such that
\[
\check{c}\leq \lambda_{k}( \mathcal{P}^{-1}\mathcal{H}( [f] ) ) \leq  \hat{c},
\] for $k=1,2,\dots,n.$

\begin{lemma}\cite[Theorem 4.5.9 (Ostrowski)]{horn_johnson_1990}\label{lemma:Ostrowski}
Let $A_m,W_m$ be $m \times m$ matrices. Suppose $A_m$ is Hermitian and $W_m$ is nonsingular. Let the eigenvalues of $A_m$ and $W_m W_m^*$ be arranged in an increasing order. For each $k=1,2,\dots,m,$ there exists a positive real number $\theta_k$ such that $\lambda_1(W_m W_m^*) \leq \theta_k \leq \lambda_m(W_m W_m^*)$ and
\begin{equation*}
\lambda_k(W_m A_m W_m^*) = \theta_k \lambda_k(A_m).
\end{equation*}
\end{lemma}

With Lemma \ref{lemma:Ostrowski}, the following results can be easily derived:

\begin{theorem}\label{thm:practical_MINRES_convergence}
Let $f \in L^1([-\pi, \pi]^d)$ and let $f=\mathrm{Re}(f)+ \mathbf{i} \mathrm{Im}(f)$, where $\mathrm{Re}(f)$ and $\mathrm{Im}(f)$ are real-valued functions with $\mathrm{Re}(f)$ essentially positive. Then, the eigenvalues of $\mathcal{P}^{-1}  Y_{n}T_{{n}}[f] $ lie in $[-\hat{c}(1+\epsilon), -\check{c}] \cup [\check{c}, \hat{c}(1+\epsilon)]$, where
$
\epsilon<\operatorname{essup}_{\boldsymbol{\theta}  \in [-\pi, \pi]^d}\left|\frac{\mathrm{Im}(f)(\boldsymbol{\theta}) }{\mathrm{Re}(f)(\boldsymbol{\theta}) }\right|.
$
\end{theorem}

\begin{proof}
Note that
\begin{eqnarray*}
&&\mathcal{P}^{-1/2} Y_n T_{n}[f] \mathcal{P}^{-1/2}  \\
&=& \mathcal{P}^{-1/2} \mathcal{H}([f])^{1/2} \mathcal{H}([f])^{-1/2} Y_n T_{n}[f] \mathcal{H}([f])^{-1/2} \mathcal{H}([f])^{1/2} \mathcal{P}^{-1/2}.
\end{eqnarray*} 
From Lemma \ref{lemma:Ostrowski} and Corollary \ref{cor:ideal_MINRES_convergence}, we know that, for each $k=1,2,\dots, n$, there exists a positive real number $\theta_k$ such that \[
\check{c}\leq\lambda_{\min}(\mathcal{P}^{-1/2} \mathcal{H}([f]) \mathcal{P}^{-1/2}) \leq \theta_k \leq \lambda_{\max}(\mathcal{P}^{-1/2} \mathcal{H}([f]) \mathcal{P}^{-1/2}) \leq \hat{c}
\] 
and
\begin{equation*}
    \lambda_k(\mathcal{P}^{-1/2} Y_n \mathcal{H}([f]) \mathcal{P}^{-1/2} ) = \theta_k \lambda_k(\mathcal{H}([f])^{-1/2} Y_n T_n[f] \mathcal{H}([f])^{-1/2}).
\end{equation*}
Recalling from Lemma \ref{lemma:pestana} that $\lambda_k(\mathcal{H}([f])^{-1/2} Y_n T_n[f] \mathcal{H}([f])^{-1/2})$ lies in $[-1-\epsilon, -1] \cup [1, 1+\epsilon]$, where
$
\epsilon<\operatorname{essup}_{\boldsymbol{\theta}  \in [-\pi, \pi]^d}\left|\frac{\mathrm{Im}(f)(\boldsymbol{\theta}) }{\mathrm{Re}(f)(\boldsymbol{\theta}) }\right|,
$ the proof is complete.
\end{proof}

The following corollary is a consequence of Theorem \ref{thm:practical_MINRES_convergence}.

\begin{corollary}\label{cor:practical_MINRES_convergence}
Let $f \in L^1([-\pi, \pi]^d)$ satisfy the assumptions made in Lemma \ref{lemma:pestana}. Then, the preconditioned MINRES method for the system $\mathcal{P}^{-1}Y_n T_{{n}}[f]$ has a convergence rate independent of $\mathbf{n}$, i.e., the residuals generated by the MINRES method satisfy
$
\frac{\|\mathbf{r}^{(k)}\|_{2}}{\|\mathbf{r}_n^{(0)}\|_{2}}\leq 2\omega^{[k]},
$
where $\mathbf{r}^{(k)} =  \mathcal{P}^{-1}Y_{n} \mathbf{b} - \mathcal{P}^{-1}  Y_{n} T_{{n}}[f]\mathbf{\tilde{u}_* }^{(k)}$, $\mathbf{\tilde{u}_* }^{(k)}$ denotes the $k$-th iterate by MINRES, $\mathbf{\tilde{u}_* }^{(0)}$ denotes an arbitrary initial guess, and $\omega$ is a constant independent of $\mathbf{n}$ defined as follows
$$
\omega := \sqrt{\frac{\hat{c}(1+\epsilon) - \check{c} }{ \hat{c}(1+\epsilon) + \check{c} }} \in \left(\sqrt{\frac{\hat{c} - \check{c} }{ \hat{c} + \check{c} }}, 1 \right) \subset (0,1),
$$
with $\epsilon<\operatorname{essup}_{\boldsymbol{\theta}  \in [-\pi, \pi]^d}\left|\frac{\mathrm{Im}(f)(\boldsymbol{\theta}) }{\mathrm{Re}(f)(\boldsymbol{\theta}) }\right|$.
\end{corollary}

To demonstrate the applicability of the preconditioning approach described in Corollary \ref{cor:practical_MINRES_convergence}, we will show in the subsequent subsections that a suitable choice of $\mathcal{P}$ is the multilevel Tau preconditioner defined in \eqref{eqn:main_krecondition_second} for solving the R.-L. fractional diffusion equation in \eqref{eq:fde}.

}

\subsection{\cred{Convergence analysis of the second-order scheme with $A_{n}$}}
\cred{
In this subsection, we show that the preconditioned MINRES method with $\mathcal{H}( [f] )$ for $Y_n T_{{n}}[f]$ can be applied to effectively solve the space R.-L. fractional diffusion equation \eqref{eq:fde} of interest.}

Before showing our main preconditioning result, we first provide two useful lemmas in what follows.

\begin{lemma}\label{lemma:wghtsumbdlem}
	For nonnegative numbers $\xi_i$ and positive numbers $\zeta_i$ $(1\leq i\leq m)$, it holds that
\begin{equation*}
\min\limits_{1\leq i\leq m}\frac{\xi_i}{\zeta_i}\leq\bigg(\sum\limits_{i=1}^{m}\zeta_i\bigg)^{-1}\bigg(\sum\limits_{i=1}^{m}\xi_i\bigg)\leq\max\limits_{1\leq i\leq m}\frac{\xi_i}{\zeta_i}.
	\end{equation*}
\end{lemma}

\cred{

\begin{lemma}\cite[Theorem 1.1]{vong2019second}\label{lemma:second_fun_property}
Let $1 < \alpha <2$ and $L_m^{(\alpha)}$ being given by \eqref{eq:grunwaldmatrix_second}. Then, the generating function of $L_m^{(\alpha)}$ is given by
\begin{equation}\label{eqn:gen_fun_Vong}
g_{\alpha}(\theta)=\left\{\begin{array}{cl}
-\left(2 \sin \frac{\theta}{2}\right)^\alpha\left[ \frac{\alpha}{2}\cos \left(\frac{\alpha}{2}(\theta-\pi)-\theta\right) +   \frac{2-\alpha}{2} \cos \left( \frac{\alpha}{2}(\theta -\pi) \right) \right. & \\
\left.-\mathbf{i} \frac{\alpha}{2}\sin \left(\frac{\alpha}{2}(\theta-\pi)-\theta\right) +   \frac{2-\alpha}{2} \sin \left( \frac{\alpha}{2}(\theta -\pi) \right)   \right], & \theta \in[0, \pi), \\
-\left(2 \sin \frac{\theta}{2}\right)^\alpha\left[ \frac{\alpha}{2}\cos \left(\frac{\alpha}{2}(\theta+\pi)-\theta\right) +   \frac{2-\alpha}{2} \cos \left( \frac{\alpha}{2}(\theta +\pi) \right) \right. & \\
\left.-\mathbf{i} \frac{\alpha}{2}\sin \left(\frac{\alpha}{2}(\theta+\pi)-\theta\right) +   \frac{2-\alpha}{2} \sin \left( \frac{\alpha}{2}(\theta +\pi) \right)   \right], & \theta \in(-\pi, 0) .
\end{array}\right.
\end{equation}
Also, 
$      \operatorname{essup}_{\theta \in[-\pi, \pi]} \left|\frac{\mathrm{Im}(g_{\alpha})(\theta) }{\mathrm{Re}(g_{\alpha})(\theta) }\right| = \left |\tan \left( \frac{\alpha}{2} \pi \right)\right |. 
$
\end{lemma}
}

The following lemma guarantees the positive definiteness of $\mathcal{H}(  A_{n})$.

\begin{lemma}\label{lemma:PS_SPD_second}
Let $A_{n}$ be the matrix defined in \eqref{eqn:main_matrix_second}. Then, the matrix $\mathcal{H}(  A_{n})=(A_{n}+A_{n}^{\top} ) / 2$ is symmetric positive definite.
\end{lemma}
\begin{proof}
    Since $\mathrm{Re} \left( g_{\alpha_i}(\theta_i) + g_{\alpha_i}(-\theta_i) \right)  $ is essentially positive by \eqref{eqn:gen_fun_Vong}, we know that $\mathcal{H}\left(L_{n_i}^{\left(\alpha_i\right)} \right)=T_{n_i}[\mathrm{Re} \left( g_{\alpha_i}(\theta_i) + g_{\alpha_i}(-\theta_i) \right)  ]$ is symmetric positive definite (see for example \cite{MR2376196,MR2108963}).
    Thus, knowing that
    	\begin{equation*}
		\mathcal{H}(  A_{n})= \nu I_{n} + \sum_{i=1}^d\left(v_{i,+}+v_{i,-}\right) I_{n_i^{-}} \otimes  	\mathcal{H}\left(L_{n_i}^{\left(\alpha_i\right)} \right) \otimes I_{n_i^{+}}, 
		\end{equation*}
    $\mathcal{H}(  A_{n})$ is also symmetric positive definite and the proof is complete.
\end{proof}

\cred{
\begin{proposition}\label{prop:main_second}
Let $f_{\boldsymbol{\alpha}}(\boldsymbol{\theta}) = \mathrm{Re}(f_{\boldsymbol{\alpha}}) + \mathbf{i} \mathrm{Im}(f_{\boldsymbol{\alpha}})$ be defined in \eqref{eqn:main_gen_function_second}. Then,
\[
\operatorname{essup}_{\boldsymbol{\theta} \in[-\pi, \pi]^d}\left|\frac{ \mathrm{Im}(f_{\boldsymbol{\alpha}})(\boldsymbol{\theta}) }{ \mathrm{Re}(f_{\boldsymbol{\alpha}}) (\boldsymbol{\theta}) }\right| \leq \max\limits_{1\leq i\leq d} \frac{  |d_{i,+}-d_{i,-}| }{d_{i,+}+d_{i,-}} \left|	\tan \left( \frac{\alpha_i}{2} \pi\right) \right|.
\]
\end{proposition}
\begin{proof}
From Lemma \ref{lemma:second_fun_property}, we know that, for each $\theta_i \in (-\pi, \pi)/\{0\}$,
\begin{eqnarray}\label{ineq:imre_second}
    \left| \frac{ \mathrm{Im}(  v_{i,+}g_{\alpha_i}(\theta_i) + v_{i,-}g_{\alpha_i}(-\theta_i)  )  }{ \mathrm{Re}(  v_{i,+}g_{\alpha_i}(\theta_i) + v_{i,-}g_{\alpha_i}(-\theta_i))  } \right| \leq  \frac{| v_{i,+}-v_{i,-} |}{v_{i,+}+v_{i,-} } \left|\tan \left( \frac{\alpha_i}{2} \pi \right) \right|.
\end{eqnarray} 
Thus, we have
   \begin{eqnarray*}
    \left|\frac{ \mathrm{Im}(f_{\boldsymbol{\alpha}}) }{ \mathrm{Re}(f_{\boldsymbol{\alpha}}) }\right| &=& \frac{| \mathrm{Im}(\nu + \sum_{i=1}^d ( v_{i,+}g_{\alpha_i}(\theta_i) + v_{i,-}g_{\alpha_i}(-\theta_i)))  |}{| \mathrm{Re}(\nu + \sum_{i=1}^d ( v_{i,+}g_{\alpha_i}(\theta_i) + v_{i,-}g_{\alpha_i}(-\theta_i)))  |} \\
  &=& \frac{| \mathrm{Im}(\sum_{i=1}^d ( v_{i,+}g_{\alpha_i}(\theta_i) + v_{i,-}g_{\alpha_i}(-\theta_i)))  |}{| \nu + \mathrm{Re}( \sum_{i=1}^d ( v_{i,+}g_{\alpha_i}(\theta_i) + v_{i,-}g_{\alpha_i}(-\theta_i)))  |} \\
    &\leq& \frac{| \mathrm{Im}(\sum_{i=1}^d ( v_{i,+}g_{\alpha_i}(\theta_i) + v_{i,-}g_{\alpha_i}(-\theta_i)))  |}{|  \mathrm{Re}( \sum_{i=1}^d ( v_{i,+}g_{\alpha_i}(\theta_i) + v_{i,-}g_{\alpha_i}(-\theta_i)))  |} \\
    &\leq& \frac{ \sum_{i=1}^d |\mathrm{Im}(   v_{i,+}g_{\alpha_i}(\theta_i) + v_{i,-}g_{\alpha_i}(-\theta_i) )  |}{|  \mathrm{Re}( \sum_{i=1}^d ( v_{i,+}g_{\alpha_i}(\theta_i) + v_{i,-}g_{\alpha_i}(-\theta_i)))  |} \\
  &=& \frac{ \sum_{i=1}^d |\mathrm{Im}(   v_{i,+}g_{\alpha_i}(\theta_i) + v_{i,-}g_{\alpha_i}(-\theta_i) )  |}{ \sum_{i=1}^d | \mathrm{Re}(  v_{i,+}g_{\alpha_i}(\theta_i) + v_{i,-}g_{\alpha_i}(-\theta_i))  |}.
    \end{eqnarray*}

Clearly, for each $\theta_i \in (-\pi, \pi)/\{0\}$, $|\mathrm{Im}(   v_{i,+}g_{\alpha_i}(\theta_i) + v_{i,-}g_{\alpha_i}(-\theta_i) )  |$ is nonnegative and $|\mathrm{Re}(   v_{i,+}g_{\alpha_i}(\theta_i) + v_{i,-}g_{\alpha_i}(-\theta_i) )  |$ is positive by \cite[Theorem 2]{Tian2015}. Thus, Lemma \ref{lemma:wghtsumbdlem} is applicable to estimating $\left|{ \mathrm{Im}(f_{\boldsymbol{\alpha}}) }/{ \mathrm{Re}(f_{\boldsymbol{\alpha}}) }\right|$. Thus, we have

 \begin{eqnarray*}
    \left|\frac{ \mathrm{Im}(f_{\boldsymbol{\alpha}}) }{ \mathrm{Re}(f_{\boldsymbol{\alpha}}) }\right|  &\leq& \max\limits_{1\leq i\leq d}\frac{ | \mathrm{Im}(  v_{i,+}g_{\alpha_i}(\theta_i) + v_{i,-}g_{\alpha_i}(-\theta_i))   |}{| \mathrm{Re}(  v_{i,+}g_{\alpha_i}(\theta_i) + v_{i,-}g_{\alpha_i}(-\theta_i))  |}  \\
  &\leq& \max\limits_{1\leq i\leq d} \frac{  |v_{i,+}-v_{i,-}| }{v_{i,+}+v_{i,-}} \left|	\tan \left( \frac{\alpha_i}{2} \pi \right) \right|\\
    &=& \max\limits_{1\leq i\leq d} \frac{  |d_{i,+}-d_{i,-}| }{d_{i,+}+d_{i,-}} \left|	\tan \left( \frac{\alpha_i}{2} \pi \right) \right|,
    \end{eqnarray*}
    where the last inequality is due to \eqref{ineq:imre_second}. The proof is complete.
\end{proof}
}

Combining Lemma \ref{lemma:pestana} and Proposition \ref{prop:main_second}, the following proposition follows.

\begin{proposition}\label{proposition:pestana}
Let $A_{n}$ be the matrix defined in \eqref{eqn:main_matrix_second} and let $\mathcal{H}(  A_{n})=(A_{n}+A_{n}^{\top} ) / 2$. Then, the eigenvalues of $\mathcal{H}(  A_{n})^{-1}  Y_{n} A_{n}$ lie in $[-(1+\epsilon), -1] \cup [1, 1+\epsilon]$, where
$$
\epsilon < \max\limits_{1\leq i\leq d} \frac{  |d_{i,+}-d_{i,-}| }{d_{i,+}+d_{i,-}} \left|	\tan \left( \frac{\alpha_i}{2}\pi \right) \right|.
$$
\end{proposition}

\cred{Given that the preconditioner $\mathcal{H}(A_n)$, while effective, cannot generally be efficiently implemented, we shall henceforth direct our attention to the proposed practical preconditioner, $P_n$. This discussion aims to illustrate how $\mathcal{H}(A_n)$ can serve as a foundational blueprint, thereby enabling the advancement of further preconditioner development.}

\cred{
\begin{lemma}\cite[Lemma 3]{Tian2015}\label{lemma:property_w_2nd}
The coefficients in \eqref{eq:grunwaldmatrix_second} satisfy the following properties for $1<\alpha_i \leq 2$,
$$
\left\{\begin{array}{l}
w_0^{(\alpha_i)}=\frac{\alpha_i}{2},~  w_1^{(\alpha_i)}=\frac{2-\alpha_i-\alpha_i^2}{2}<0,~ w_2^{(\alpha_i)}=\frac{\alpha_i\left(\alpha_i^2+\alpha_i-4\right)}{4}, \\
1 \geq w_0^{(\alpha_i)} \geq w_3^{(\alpha_i)} \geq w_4^{(\alpha_i)} \geq \cdots \geq 0, \\
\sum_{k=0}^{\infty} w_k^{(\alpha_i)}=0, \sum_{k=0}^{n_i} w_k^{(\alpha_i)}<0,~{n_i} \geq 2.
\end{array}\right.
$$
\end{lemma}
}

\cred{
The following lemma guarantees the positive definiteness of $P_{n}$.

\begin{lemma}\label{lemma:PS_SPD_second}
The matrix $P_{n}$ defined in \eqref{eqn:main_krecondition_second} is symmetric positive definite.
\end{lemma}
\begin{proof}
From \eqref{eq:grunwaldmatrix_second}, the first column of $\mathcal{H}\left({L}_{n_i}^{\left(\alpha_i\right)} \right)$ is
$$-\left[2w_1^{(\alpha_i)}, w_0^{(\alpha_i)}+w_2^{(\alpha_i)}, w_3^{(\alpha_i)}, \dots, w_{n_i}^{(\alpha_i)}\right].$$
Consider a function defined as $f\left(\alpha_i\right)=w_0^{\left(\alpha_i\right)}+w_2^{\left(\alpha_i\right)}=\frac{1}{4}\left(2 \alpha_i+\alpha_i^3+\alpha_i^2-4 \alpha_i\right)$. For values of $\alpha_i$ within the interval $(1,2)$, it is evident that the function $f\left(\alpha_i\right)$ increases within this range, signifying that $f\left(\alpha_i\right)\geq f(1)=0$. Consequently, it can be inferred that $w_0^{(\alpha_i)} + w_2^{(\alpha_i)}$ is also greater than zero.
Taking advantage of \eqref{sigmicomp}, the $j$-th eigenvalue of $-\tau\left(\mathcal{H}\left({L}_{n_i}^{\left(\alpha_i\right)} \right)\right)$ can be expressed as
$$
{
\begin{aligned}
\lambda_j & =2 w_1^{(\alpha_i)}+2\left(w_0^{(\alpha_i)}+w_2^{(\alpha_i)}\right) \cos \left(\frac{\pi j}{n_i+1}\right)+2 \sum_{k=3}^{n_i} w_k^{(\alpha_i)} \cos \left(\frac{\pi j(k-1)}{n_i+1}\right) \\
& \leq 2 w_1^{(\alpha_i)}+2\left(w_0^{(\alpha_i)}+w_2^{(\alpha_i)}\right)+2 \sum_{k=3}^{n_i} w_k^{(\alpha_i)} \\
& =2 \sum_{k=0}^{n_i} w_k^{(\alpha_i)}<0 .
\end{aligned}
}
$$
Therefore, $\tau\left(\mathcal{H}\left({L}_{n_i}^{\left(\alpha_i\right)} \right)\right)$ is positive definite, which further implies that $P_n$ is positive definite. The proof is complete.
\end{proof}
}
\cred{
Lemma \ref{lemma:propertyeigtau} plays a crucial role in our convergence analysis.

\begin{lemma}\cite{Huang_2022}\label{lemma:propertyeigtau}
  Let $T_n=\left[t_{|i-j|}\right]$ be a symmetric Toeplitz matrix and $\tau\left(T_n\right)$ be the corresponding $\tau$ matrix. If the entries of $T_n$ are equipped with the following properties,
$$
t_0>0, t_1<t_2<t_3<\cdots<t_{n-1}<0 \quad \text{and} \quad t_0+2 \sum_{i=1}^n t_i>0 \text {, }
$$
or
$$
t_0<0, t_1>t_2>t_3>\cdots>t_{n-1}>0 \quad \text{and}\quad t_0+2 \sum_{i=1}^n t_i<0,
$$
the eigenvalues of matrix $\tau\left(T_n\right)^{-1} T_n$ satisfy
$$
1 / 2<\lambda\left(\tau\left(T_n\right)^{-1} T_n\right)<3 / 2.
$$  
\end{lemma}
}

\cred{
\begin{lemma}\label{lemma:tau_on_fractiona_mat_second}
Let $L_{n_i}^{(\alpha_i)}$ be the Toeplitz matrix defined in \eqref{eq:grunwaldmatrix_second}. Then, the eigenvalues of $\tau( \mathcal{H}( L_{n_i}^{(\alpha_i)}  )  )^{-1} \mathcal{H}( L_{n_i}^{(\alpha_i)}  )  $ lie in $(1/2, 3/2)$ for $\alpha_i \in (1, 2)$.
\end{lemma}
\begin{proof}   
Noting that $1<\alpha_i<2$, we have
$$
w_0^{(\alpha_i)}+w_2^{(\alpha_i)}-w_3^{(\alpha_i)}=\frac{1}{12}\alpha_i^2(\alpha_i-1)(\alpha_i+4)>0,
$$
which implies $w_0^{(\alpha_i)}+w_2^{(\alpha_i)} > w_3^{(\alpha_i)}$. 
From the property of {$w_{k}^{(\alpha_i)}$} in Lemma \ref{lemma:property_w_2nd}, the coefficients {$w_k^{(\alpha_i)}$} satisfy 
$$
2w_1^{(\alpha_i)}<0, w_0^{(\alpha_i)}+w_2^{(\alpha_i)}>w_3^{(\alpha_i)}>w_4^{(\alpha_i)}>\cdots>0, \quad {2\sum_{k=0}^{n_i} w_k^{(\alpha_i)}<0}.
$$
From Lemma \ref{lemma:propertyeigtau}, we complete the proof.
\end{proof}
}
The following proposition indicates that $P_{n}$ and $\mathcal{H}( A_{n})$ are spectrally equivalent.

\begin{proposition}\label{prop:eigen_S_second}
Let $A_{n}, P_{n}$ be the matrices defined in \eqref{eqn:main_matrix_second} and \eqref{eqn:main_krecondition_second}, respectively. Then, the eigenvalues of $P_{n}^{-1}   \mathcal{H}( A_{n})$ lie in $(1/2, 3/2 )$.
\end{proposition}
\begin{proof}
    Let $(\lambda, \mathbf{w})$ be an arbitrary eigenpair of $P_{n}^{-1}   \mathcal{H}( A_{n})$. Then, it holds
    \begin{eqnarray*}
        \lambda &=& \frac{ \mathbf{w}^*    \mathcal{H}( A_{n}) \mathbf{w} }{\mathbf{w}^* P_{n} \mathbf{w}} \\
        &=&\frac{ \mathbf{w}^*    \left( \nu I_{n} + \sum_{i=1}^d\left(v_{i,+}+v_{i,-}\right) I_{n_i^{-}} \otimes 	\mathcal{H}\left(L_{n_i}^{\left(\alpha_i\right)} \right) \otimes I_{n_i^{+}}  \right) \mathbf{w} }{ \mathbf{w}^*    \left( \nu I_{n} + \sum_{i=1}^d\left(v_{i,+}+v_{i,-}\right) I_{n_i^{-}} \otimes \tau\left(	\mathcal{H}\left(L_{n_i}^{\left(\alpha_i\right)} \right)\right) \otimes I_{n_i^{+}}  \right) \mathbf{w} }.
    \end{eqnarray*}

    Now, combining Lemma \ref{lemma:tau_on_fractiona_mat_second} with the Rayleigh quotient theorem, we have
    \begin{eqnarray*}
        \frac{1}{2} &\leq&\lambda_{\min} \left( \tau( \mathcal{H}( L_{n_i}^{(\alpha_i)}  )  )^{-1}
        \mathcal{H}( L_{n_i}^{(\alpha_i)}  )  \right)  \\
        &\leq& 
\frac{ \mathbf{y}^*    \mathcal{H}( L_{n_i}^{(\alpha_i)} ) \mathbf{y} }{\mathbf{y}^* \tau( \mathcal{H}( L_{n_i}^{(\alpha_i)}  )  ) \mathbf{y}} \\
&\leq&  \lambda_{\max} \left( \tau( \mathcal{H}( L_{n_i}^{(\alpha_i)}  )  )^{-1} \mathcal{H}( L_{n_i}^{(\alpha_i)}  )  \right) \leq \frac{3}{2},
    \end{eqnarray*}
    for any nonzero vector $\mathbf{y}$.

    Then, we have
    \begin{eqnarray*}
        \frac{1}{2} &=& \frac{1}{2} \cdot \frac{ \mathbf{w}^*    \left( \sum_{i=1}^d\left(v_{i,+}+v_{i,-}\right) I_{n_i^{-}} \otimes 	\tau\left(	\mathcal{H}\left(L_{n_i}^{\left(\alpha_i\right)} \right)\right)  \otimes I_{n_i^{+}}  \right) \mathbf{w} }{ \mathbf{w}^*    \left(  \sum_{i=1}^d\left(v_{i,+}+v_{i,-}\right) I_{n_i^{-}} \otimes \tau\left(	\mathcal{H}\left(L_{n_i}^{\left(\alpha_i\right)} \right)\right) \otimes I_{n_i^{+}}  \right) \mathbf{w} } \\
        &\leq& \frac{ \mathbf{w}^*    \left( \sum_{i=1}^d\left(v_{i,+}+v_{i,-}\right) I_{n_i^{-}} \otimes 	\mathcal{H}\left(L_{n_i}^{\left(\alpha_i\right)} \right) \otimes I_{n_i^{+}}  \right) \mathbf{w} }{ \mathbf{w}^*    \left(  \sum_{i=1}^d\left(v_{i,+}+v_{i,-}\right) I_{n_i^{-}} \otimes \tau\left(	\mathcal{H}\left(L_{n_i}^{\left(\alpha_i\right)} \right)\right) \otimes I_{n_i^{+}}  \right) \mathbf{w} } \\
        &\leq& \frac{3}{2} \cdot \frac{ \mathbf{w}^*    \left( \sum_{i=1}^d\left(v_{i,+}+v_{i,-}\right) I_{n_i^{-}} \otimes 	\tau\left(	\mathcal{H}\left(L_{n_i}^{\left(\alpha_i\right)} \right)\right)  \otimes I_{n_i^{+}}  \right) \mathbf{w} }{ \mathbf{w}^*    \left(  \sum_{i=1}^d\left(v_{i,+}+v_{i,-}\right) I_{n_i^{-}} \otimes \tau\left(	\mathcal{H}\left(L_{n_i}^{\left(\alpha_i\right)} \right)\right) \otimes I_{n_i^{+}}  \right) \mathbf{w} } = \frac{3}{2}.
    \end{eqnarray*}

    By Lemma \ref{lemma:wghtsumbdlem}, it follows that
    \begin{eqnarray*}
        \frac{1}{2}&=&\min{\left\{1,\frac{1}{2}\right\} } \\
        &\leq& \frac{ \mathbf{w}^*    \left( \nu I_{n} + \sum_{i=1}^d\left(v_{i,+}+v_{i,-}\right) I_{n_i^{-}} \otimes 	\mathcal{H}\left(L_{n_i}^{\left(\alpha_i\right)} \right) \otimes I_{n_i^{+}}  \right) \mathbf{w} }{ \mathbf{w}^*    \left( \nu I_{n} + \sum_{i=1}^d\left(v_{i,+}+v_{i,-}\right) I_{n_i^{-}} \otimes \tau\left(	\mathcal{H}\left(L_{n_i}^{\left(\alpha_i\right)} \right)\right) \otimes I_{n_i^{+}}  \right) \mathbf{w} }\\
        &\leq&\max{\left\{ 1,\frac{3}{2}\right\} }=\frac{3}{2}, 
    \end{eqnarray*}
    which implies $\lambda \in (\frac{1}{2},\frac{3}{2})$. The proof is complete.
\end{proof}

The following theorem shows the preconditioning effectiveness of $P_{n}$ for $Y_{n} A_{n}$, \cred{as a consequence of Theorem \ref{thm:practical_MINRES_convergence} and Corollary \ref{cor:practical_MINRES_convergence}, by letting $\check{c}=1/2$ and $\hat{c}=3/2$.}

	\begin{theorem}\label{thm:main_theorem_second}
			Let $A_{n}, P_{n}$ be the matrices defined in \eqref{eqn:main_matrix_second} and \eqref{eqn:main_krecondition_second}, respectively. Then, the eigenvalues of $P_{n}^{-1}  Y_{n} A_{n}$ lie in $\left(-\frac{3}{2}(1+\epsilon), -\frac{1}{2}\right) \cup \left(\frac{1}{2}, \frac{3}{2}(1+\epsilon)\right)$, where
$$
\epsilon < \max\limits_{1\leq i\leq d} \frac{  |d_{i,+}-d_{i,-}| }{d_{i,+}+d_{i,-}} \left|	\tan \left( \frac{\alpha_i}{2}\pi \right) \right|.
$$
	\end{theorem}

\begin{corollary}
    The preconditioned MINRES method for the system $Y_n A_n \mathbf{x}=Y_n \mathbf{b}$ in \eqref{eqn:preconditioned} has a convergence rate independent of $\mathbf{n}$, i.e., the residuals generated by the MINRES method satisfy
    $$
    \frac{\|\mathbf{r}^{(k)}\|_{2}}{\|\mathbf{r}_n^{(0)}\|_{2}}\leq 2\omega^{[k]},
    $$
    where $\mathbf{r}^{(k)} =  P_{n}^{-1}Y_{n} \mathbf{b} - P_{n}^{-1}  Y_{n} A_{n}\mathbf{\tilde{u}_* }^{(k)}$, $\mathbf{\tilde{u}_* }^{(k)}$ denotes the $k$-th iterate by MINRES, $\mathbf{\tilde{u}_* }^{(0)}$ denotes an arbitrary initial guess, and $\omega$ is a constant independent of $\mathbf{n}$ defined as follows
    $$
    \omega := \sqrt{\frac{2 + 3\epsilon }{4 + 3\epsilon}} \in \left(\sqrt{\frac{1}{2}}, 1 \right) \subset (0,1),
    $$
    with $\epsilon$ given by \cred{Theorem \ref{thm:main_theorem_second}}.
\end{corollary}

\subsection{\cred{Convergence analysis of the first-order scheme with $\widetilde{A}_{n}$}}
 
\cred{

In this subsection, we show that the MINRES method, when used with the preconditioner $\widetilde{P}_n$, also achieves optimal convergence in the first-order case with the shifted Gr\"unwald scheme. However, as discussed in Remark \ref{remark:first_order_precon}, the matrix $\widetilde{P}_n$ closely resembles the preconditioner proposed by \cite{Lin_etc_2023}, with only minimal differences. Consequently, we provide the relevant results for $\widetilde{P}_n$ below without further details.

\begin{lemma}\cite[Lemmas 3.2 \& 3.3]{PangSun16}\label{lemma:PangSun}
 Let $1 < \alpha <2$ and $\widetilde{L}_m^{(\alpha)}$ being given by \eqref{eq:grunwaldmatrix_first}. Then, the generating function of $\widetilde{L}_m^{(\alpha)}$ is given by
\begin{equation*}\label{eqn:gen_fun_PangSun}
\widetilde{g}_{\alpha}(\theta)=\left\{\begin{array}{cl}
    -\left(2 \sin \frac{\theta}{2}\right)^\alpha\left[ \cos \left(\frac{\alpha}{2}(\pi-\theta)+\theta\right)-\mathbf{i} \sin \left(\frac{\alpha}{2}(\pi-\theta)+\theta\right)\right], & \theta \in[0, \pi), \\
    -\left(2 \sin \frac{-\theta}{2}\right)^\alpha\left[ \cos \left(\frac{\alpha}{2}(\pi+\theta)-\theta\right)+\mathbf{i}  \sin \left(\frac{\alpha}{2}(\pi+\theta)-\theta\right)\right], & \theta \in(-\pi, 0) .
\end{array}\right.
\end{equation*}
Also,
$
\operatorname{essup}_{\theta \in[-\pi, \pi]} \left|\frac{\mathrm{Im}(\widetilde{g}_{\alpha})(\theta) }{\mathrm{Re}(\widetilde{g}_{\alpha})(\theta) }\right| = \left |\tan \left( \frac{\alpha}{2} \pi \right)\right |. 
$
\end{lemma}

Similar to the second-order case with Lemma \ref{lemma:second_fun_property} replaced by Lemma \ref{lemma:PangSun}, we have the following results explaining the preconditioning effectiveness of the ideal preconditioner $\mathcal{H}(  \widetilde{A}_{n})$ for $Y_n \widetilde{A}_n \mathbf{x} = Y_n \mathbf{b}$, where $\mathcal{H}(  \widetilde{A}_{n})$ is precisely the ideal preconditioner proposed in \cite{Pestana2019}.

\begin{lemma}\label{lemma:PS_SPD}
Let $\widetilde{A}_{n}$ be the matrix defined in \eqref{eqn:main_matrix_first}. Then, the matrix $\mathcal{H}(  \widetilde{A}_{n})=(\widetilde{A}_{n}+\widetilde{A}_{n}^{\top} ) / 2$ is symmetric positive definite.
\end{lemma}

\begin{proposition}\label{prop:main_first}
Let $\widetilde{f}_{\boldsymbol{\alpha}}(\boldsymbol{\theta}) = \mathrm{Re}(\widetilde{f}_{\boldsymbol{\alpha}}) + \mathbf{i} \mathrm{Im}(\widetilde{f}_{\boldsymbol{\alpha}})$ be defined in \eqref{eqn:main_gen_function_first}. Then,
\[
\operatorname{essup}_{\boldsymbol{\theta} \in[-\pi, \pi]^d}\left|\frac{ \mathrm{Im}(\widetilde{f}_{\boldsymbol{\alpha}})(\boldsymbol{\theta}) }{ \mathrm{Re}(\widetilde{f}_{\boldsymbol{\alpha}}) (\boldsymbol{\theta}) }\right| \leq \max\limits_{1\leq i\leq d} \frac{  |d_{i,+}-d_{i,-}| }{d_{i,+}+d_{i,-}} \left|	\tan \left( \frac{\alpha_i}{2} \pi\right) \right|.
\]
\end{proposition}

Combining Lemma \ref{lemma:pestana} and Proposition \ref{prop:main_first}, the following proposition and corollary follow.

\begin{proposition}\label{proposition:pestana_first}
Let $\widetilde{A}_{n}$ be the matrix defined in \eqref{eqn:main_matrix_first} and let $\mathcal{H}(  \widetilde{A}_{n})=(\widetilde{A}_{n}+\widetilde{A}_{n}^{\top} ) / 2$. Then, the eigenvalues of $\mathcal{H}(  \widetilde{A}_{n})^{-1}  Y_{n} \widetilde{A}_{n}$ lie in $[-(1+\epsilon), -1] \cup [1, 1+\epsilon]$, where
$$
\epsilon < \max\limits_{1\leq i\leq d} \frac{  |d_{i,+}-d_{i,-}| }{d_{i,+}+d_{i,-}} \left|	\tan \left( \frac{\alpha_i}{2}\pi \right) \right|.
$$
\end{proposition}

\begin{corollary}\label{coro:pestana_first}
The preconditioned MINRES method for the system $Y_{n}\widetilde{A}_n\mathbf{x}=Y_n\mathbf{b}$ in \eqref{eqn:preconditioned} has a convergence rate independent of $\mathbf{n}$, i.e., the residuals generated by the MINRES method satisfy
$
\frac{\|\mathbf{r}^{(k)}\|_{2}}{\|\mathbf{r}_n^{(0)}\|_{2}}\leq 2\omega^{[k]},
$
where $\mathbf{r}^{(k)} =  \mathcal{H}(  \widetilde{A}_{n})^{-1}Y_{n} \mathbf{b} - \mathcal{H}(  \widetilde{A}_{n})^{-1}  Y_{n} \widetilde{A}_{n}\mathbf{\tilde{u}_* }^{(k)}$, $\mathbf{\tilde{u}_* }^{(k)}$ denotes the $k$-th iterate by MINRES, $\mathbf{\tilde{u}_* }^{(0)}$ denotes an arbitrary initial guess, and $\omega$ is a constant independent of $\mathbf{n}$ defined as follows
$
\omega := \sqrt{\frac{\epsilon }{2 + \epsilon}} \in (0,1),
$
with $\epsilon$ given by Proposition \ref{proposition:pestana_first}.
\end{corollary}

Corollary \ref{coro:pestana_first} accounts for the numerically observed superior preconditioning of $\mathcal{H}(  \widetilde{A}_{n})$ for solving the concerned space fractional diffusion equation \cite{Pestana2019} when the first-order scheme is used, as it clearly shows that the number of iterations MINRES required to converge is independent of $\mathbf{n}$ in this case. However, we stress that $\mathcal{H}(  \widetilde{A}_{n})$ as a preconditioner is not effective in general for nonsymmetric multilevel Toeplitz systems, as illustrated for example in \cite[Example 5.1]{Pestana2019}. Along this research line, we direct readers to \cite{Hon_SC_Wathen} where various effective absolute value preconditioning techniques were specifically designed for general nonsymmetric Toeplitz systems.

We now turn our attention to the practical preconditioner $\widetilde{P}_n$, which approximates $\mathcal{H}(  \widetilde{A}_{n})$.

\begin{lemma}\cite[Lemma 2.2]{Lin_etc_2023}\label{lemma:PS_SPD_first}
The matrix $\widetilde{P}_{n}$ defined in \eqref{eqn:main_krecondition_first} is symmetric positive definite.
\end{lemma}

\begin{lemma}\cite[Lemma 3.3]{Lin_etc_2023}\label{lemma:eigen_S_first}
Let $\widetilde{A}_{n}, \widetilde{P}_{n}$ be the matrices defined in \eqref{eqn:main_matrix_first} and \eqref{eqn:main_krecondition_second}, respectively. Then, the eigenvalues of $\widetilde{P}_{n}^{-1}   \mathcal{H}( \widetilde{A}_{n})$ lie in $(1/2, 3/2 )$.
\end{lemma}

The following results follow, as a consequence of Theorem \ref{thm:practical_MINRES_convergence} and Corollary \ref{cor:practical_MINRES_convergence}.

\begin{theorem}\label{thm:main_theorem_first}
Let $\widetilde{A}_{n}, \widetilde{P}_{n}$ be the matrices defined in \eqref{eqn:main_matrix_first} and \eqref{eqn:main_krecondition_first}, respectively. Then, the eigenvalues of $\widetilde{P}_{n}^{-1}  Y_{n} \widetilde{A}_{n}$ lie in $\left(-\frac{3}{2}(1+\epsilon), -\frac{1}{2}\right) \cup \left(\frac{1}{2}, \frac{3}{2}(1+\epsilon)\right)$, where
$
\epsilon < \max\limits_{1\leq i\leq d} \frac{  |d_{i,+}-d_{i,-}| }{d_{i,+}+d_{i,-}} \left|	\tan \left( \frac{\alpha_i}{2}\pi \right) \right|.
$
\end{theorem}

\begin{corollary}
    The preconditioned MINRES method for the system $Y_n \widetilde{A}_n {\mathbf{x}}=Y_n {\mathbf{b}}$ in \eqref{eqn:preconditioned} has a convergence rate independent of $\mathbf{n}$, i.e., the residuals generated by the MINRES method satisfy $\frac{\|\mathbf{r}^{(k)}\|_{2}}{\|\mathbf{r}_n^{(0)}\|_{2}}\leq 2\omega^{[k]},$
    where $\mathbf{r}^{(k)} =  \widetilde{P}_{n}^{-1}Y_{n} \mathbf{b} - \widetilde{P}_{n}^{-1}  Y_{n} \widetilde{A}_{n}\mathbf{\tilde{u}_* }^{(k)}$, $\mathbf{\tilde{u}_* }^{(k)}$ denotes the $k$-th iterate by MINRES, $\mathbf{\tilde{u}_* }^{(0)}$ denotes an arbitrary initial guess, and $\omega$ is a constant independent of $\mathbf{n}$ defined as follows 
    $$
    \omega := \sqrt{\frac{2 + 3\epsilon }{4 + 3\epsilon}} \in \left(\sqrt{\frac{1}{2}}, 1 \right) \subset (0,1),
    $$
    with $\epsilon$ given by Theorem \ref{thm:main_theorem_first}.
\end{corollary}

}

\section{Numerical examples}\label{sec:numerical}

In this section, we demonstrate the effectiveness of our proposed preconditioner against the state-of-the-art preconditioned MINRES solver proposed in \cite{Pestana2019} \cred{and the state-of-the-art preconditioned GMRES solver proposed in \cite{Lin_etc_2023}. All numerical experiments are carried out using MATLAB 8.2.0 on a Dell R640 server with dual Xeon Gold 6246R 16-Cores 3.4 GHz CPUs and 512GB RAM running Ubuntu $20.04$ LTS.} Our proposed preconditioner $P_n$ \cred{(or $\widetilde{P}_n$)} is implemented by the built-in function \textbf{dst} (discrete sine transform) in MATLAB. Furthermore, the MINRES solver is implemented using the function \textbf{minres}. We choose $x_0=(1,1,\dots,1)^{\top}/\sqrt{n}$ as our initial guess and a stopping tolerance of $10^{-8}$ based on the reduction in relative residual norms for MINRES. In the related tables, we denote by `Iter' the number of iterations for solving a linear system by an iterative solver within the given accuracy and `CPU' is the time needed for convergence measured in seconds using the MATLAB built-in functions \textbf{tic/toc}.

\cred{In what follows, we will test our preconditioner using the numerical test described in \cite[Example 5.3]{Pestana2019} for comparison purposes.} The notation $A_R:=\mathcal{H}(  A_{n})=T_n[\textrm{Re}(f_{\boldsymbol{\alpha}})]$ and $A_M:=T_n[|f_{\boldsymbol{\alpha}}|]$ are used to denote the existing ideal Toeplitz preconditioners proposed in the same work. Note that we did not compare with the absolute value circulant preconditioners proposed in \cite{doi:10.1137/140974213,Hon_SC_Wathen}, where the well-known Strang \cite{Strang1986} circulant preconditioner and the absolute value optimal \cite{doi:10.1137/0909051} circulant preconditioner could be used. It is expected that their effectiveness cannot surpass both $A_R$ and $A_M$ as studied in the numerical tests carried out in \cite{Pestana2019}, particularly in the ill-conditioned or multilevel case.

We adopt the notation MINRES-$I_n$/$P_n$/$MG(A_R)$/$MG(A_M)$ to denote the MINRES solver with $I_n$ (the identity matrix, representing the non-preconditioned case), our proposed preconditioner $P_n$, and the multigrid approximation of the state-of-the-art preconditioners $A_R$ and $A_M$, respectively. \cred{Similar notation is defined when the first-order scheme is considered.}

\begin{example}\label{example:two_dim_first}
\rm{
We begin to solve a nonsymmetric two-level Toeplitz problem, which is associated with the fractional diffusion problem stated in \eqref{eq:fde} with zero initial condition, where 
 \begin{align*}
    &d=2,~\Omega=(0,1)\times(0,1),~T=1,~d_{1,+}=2,~d_{1,-}=0.5,\\
    &d_{2,+}=0.3,~d_{2,-}=1,~f(x_1,x_2,t)=100 \sin (10 x_1) \cos (x_2)+\sin (10 t) x_1 x_2.
\end{align*}
We choose $n_1=n_2$ and $\tau=1 /\left\lceil n_1^{\alpha_1}\right\rceil$. Note that \cred{the shifted Gr\"unwald scheme is used, and} the stated CPU times and iteration counts apply to the first time step. \cred{The GMRES solver with $\textrm{restart}=20$ proposed in \cite{Lin_etc_2023} is denoted by GMRES-$\widetilde{P}_n$.}

As with \cite[Example 5.3]{Pestana2019}, it is too computationally costly to approximate \cred{$\widetilde{A}_M$} by a banded Toeplitz matrix or by using a multigrid method because computing the Fourier coefficients of \cred{$|\widetilde{f}_{\boldsymbol{\alpha}}|$} is expensive. Therefore, results are only reported for a multigrid approximation to \cred{$\widetilde{A}_R$}. The multigrid method includes four pre-smoothing and four post-smoothing steps, with a damping parameter of $0.9$. The coarsest grid has dimensions $n_1=n_2=7$.

Table \ref{Example_1_Table} shows the iteration count and CPU time for the MINRES solver using various preconditioners, compared across different orders of fractional derivatives $(\alpha_1,\alpha_2)$ and \( n=n_1  n_2 \). The findings presented in \cred{Table \ref{Example_1_Table}} highlight the following observations:

\begin{enumerate}
    \item \cred{MINRES-\( \cred{\widetilde{P}_n} \) outperforms GMRES-\( \cred{\widetilde{P}_n} \) when the matrix size is sufficiently large.}
    \item MINRES-\( \cred{\widetilde{P}_n} \) achieves iteration counts that are independent of the mesh size, establishing it as the most efficient method.
    \item The convergence of MINRES-\( \cred{\widetilde{P}_n} \) is determined by \( \max\limits_{1 \leq i \leq 2} \left| \tan \left( \frac{\alpha_i}{2}\pi \right) \right| \), rather than by each \( \left| \tan \left( \frac{\alpha_i}{2}\pi \right) \right| \) individually.
    \item As \( \alpha_1 \) and \( \alpha_2 \) approach 2, the convergence of MINRES-\( \cred{\widetilde{P}_n} \) improves due to the corresponding reduction of \( \max\limits_{1 \leq i \leq 2} \left| \tan \left( \frac{\alpha_i}{2}\pi \right) \right| \) towards zero.
\end{enumerate}

Figures \ref{fig:two_D_1_1_225}--\ref{fig:two_D_1_9_961} illustrate the eigenvalues of \( \cred{\widetilde{P}_n}^{-1}Y_n \cred{\widetilde{A}_n} \) for various values of \( \alpha_i \) and {\( n \)}, validating Theorem \ref{thm:main_theorem_first} and demonstrating the improved effectiveness of preconditioning when both \( \alpha_1 \) and \( \alpha_2 \) are close to 2.

\begin{table}[H]
\cred{
\small \caption{{Performance of MINRES-$I_n$, MINRES-$MG(\widetilde{A}_R)$, GMRES-$\widetilde{P}_n$ and MINRES-$\widetilde{P}_n$ for Example \ref{example:two_dim_first} with $d_{1,+}=2,~d_{1,-}=0.5,~d_{2,+}=0.3$, and $d_{2,-}=1$.}}
\label{Example_1_Table}
\begin{center}
\begin{tabular}{cc|cc|cc|cc|cc}
\hline
\multirow{2}{*}{$(\alpha_1,\alpha_2)$} & \multirow{2}{*}{$n$} & \multicolumn{2}{c|}{MINRES-$I_n$} & \multicolumn{2}{c|}{MINRES-$MG(\widetilde{A}_R)$}  & \multicolumn{2}{c|}{GMRES-$\widetilde{P}_n$}& \multicolumn{2}{c}{MINRES-$\widetilde{P}_n$} \\ \cline{3-10}
                                  &                        & Iter         & CPU        & Iter          & CPU         & Iter       & CPU    & Iter       & CPU       \\ \hline                    
\multirow{4}{*}{(1.1,1.1)}      
                                  & 261121                 &  $>$100           &  -           &  12            &  0.86         &9&0.40     &      12     &  0.56    \\
                                  & 1046529                &  $>$100           &  -            &  12            &  3.2           &9&1.77  &    12       &  2.05     \\
                                  & 4190209               & $>$100           &  -             &   12           &  73           &9&12.23  &     12       &  11.20       \\  & 16769025               & $>$100           &  -             &   12           &  1.2e+2          &9&63.94   &     12       & 47.21       \\\hline
\multirow{4}{*}{(1.1,1.5)}      
                                  & 261121                 & $>$100           &  -            &   22           &  1.4          &9&0.43  &  16         &  0.72      \\
                                  & 1046529                &   $>$100           &  -             &  26            & 6.4            &9&1.74  &  14         & 2.36      \\
                                  & 4190209               &   $>$100           &  -             &   30           &  1.9e$+2$          &9&12.39  & 14          & 12.61 \\       & 16769025               & $>$100           &  -             &   36           &  3.4e+2            &9&63.68 &     14       & 54.41       \\ \hline
\multirow{4}{*}{(1.1,1.9)}       
                                  & 261121                 &  $>$100           &  -            &  $>$100            & -            &9&0.40  &  14         &  0.63     \\
                                  & 1046529                &  $>$100           &  -          & $>$100             & -             &9&1.71 & 14          &  2.38     \\
                                  & 4190209               &   $>$100           &  -          &  $>$100            & -            &9&12.48 & 14          & 12.44       \\ & 16769025               & $>$100           &  -             &   $>$100           &  -           &9&63.70  &     14       & 54.21       \\ \hline
\multirow{4}{*}{(1.5,1.1)}      
                                  & 261121                 &  $>$100           &  -           &   14           &   0.95         &7&0.32  &  10         & 0.48       \\
                                  & 1046529                &   $>$100           &  -            & 14             &  3.5           &6&1.31   &  10         & 1.77     \\
                                  & 4190209               &    $>$100           &  -            &  14            & 94             &6&9.59   & 10          & 9.25    \\ & 16769025               & $>$100           &  -             &   13           &  1.4e+2           &6&53.51  &     10       & 40.07       \\ \hline
\multirow{4}{*}{(1.5,1.5)}      
                                  & 261121                 &     $>$100           &  -            &  10            & 0.69             &7&0.37 &  12         &  0.56     \\
                                  & 1046529                &  $>$100           &  -           &    8          &   2.2          &6&1.31  &  11         &   1.88    \\
                                  & 4190209               &   $>$100           &  -           &   8           &  56            &6&9.42   & 10          & 9.30    \\ & 16769025               & $>$100           &  -             &   8          &  86           &6&53.80   &     10       & 40.18      \\ \hline
\multirow{4}{*}{(1.5,1.9)}       
                                  & 261121                 &   $>$100           &  -           &   22           &    1.4          &7&0.32  &  11         &  0.52    \\
                                  & 1046529                &   $>$100           &  -           &  26            &  6.2            &6&1.30  &  11         &  1.90    \\
                                  & 4190209               &  $>$100           &  -           &   31           &  2e$+2$           &6&9.32 & 10          &  9.05      \\ & 16769025               & $>$100           &  -             &   39           &  3.8e+2           &6&53.39  &     10       & 40.08       \\ \hline
\multirow{4}{*}{(1.9,1.1)}      
                                  & 261121                 &   $>$100           &  -          &    17          &  1.1           &4&0.21   & 7          &  0.34    \\
                                  & 1046529                &   $>$100           &  -         &     17         &   4.2           &4& 1.03  & 7          &  1.30    \\
                                  & 4190209               &     $>$100           &  -            &    17          &  1e$+2$          &4&7.71  & 7          &  6.75      \\ & 16769025               & $>$100           &  -             &   16           &  1.7e+2          &4&47.09    &     7       & 29.58      \\ \hline
\multirow{4}{*}{(1.9,1.5)}      
                                  & 261121                 &   $>$100           &  -            &    15          &   0.99         &4&0.23   &  8         & 0.39      \\
                                  & 1046529                &    $>$100           &  -         &        15      &    3.8         & 4&1.01  &  8         &  1.46    \\
                                  & 4190209               &         $>$100           &  -           &       15       & 96             &4&7.89   & 8          & 7.82    \\ & 16769025               & $>$100           &  -             &   16           &  1.6e+2           &4&46.88   &     7       & 29.52      \\ \hline
\multirow{4}{*}{(1.9,1.9)}        
                                  & 261121                 &   $>$100           &  -           &        9      &  0.61           &4&0.23    &   9        &  0.43   \\
                                  & 1046529                &   $>$100           &  -            &       9       &  2.4           &4&1.03   &  9         &  1.55    \\
                                  & 4190209               &    $>$100           &  -             &            9 &  60             &4&7.80    & 9          &  8.51   \\ & 16769025               & $>$100           &  -             &   9           &  99           &4&47.20   &     9       & 36.54      \\ \hline                             
\end{tabular}
\end{center}
}
\end{table}

\begin{figure}[h!]
    \centering
     \begin{subfigure}[b]{0.8\textwidth}
    \includegraphics[width=\textwidth]{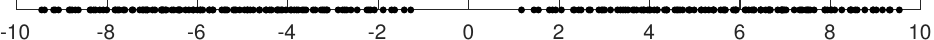}
         \caption{without preconditioner}
        \includegraphics[width=\textwidth]{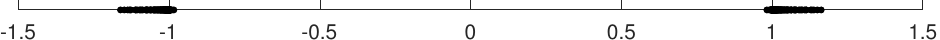}
         \caption{with $\cred{\widetilde{P}_n}$}
     \end{subfigure}
        \caption{Eigenvalues of $Y_n\cred{\widetilde{A}_n}$ with $\alpha_1=\alpha_2=1.1$ and $n=225$.}
        \label{fig:two_D_1_1_225}
\end{figure}

\begin{figure}[h!]
    \centering
     \begin{subfigure}[b]{0.8\textwidth}
    \includegraphics[width=\textwidth]{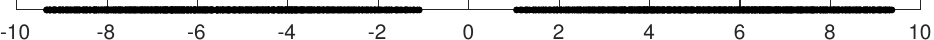}
         \caption{without preconditioner}
        \includegraphics[width=\textwidth]{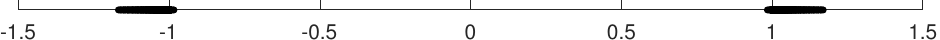}
         \caption{with $\cred{\widetilde{P}_n}$}
     \end{subfigure}
        \caption{Eigenvalues of $Y_n\cred{\widetilde{A}_n}$ with $\alpha_1=\alpha_2=1.1$ and $n=961$.}
        \label{fig:two_D_1_1_961}
\end{figure}

\begin{figure}[h!]
    \centering
     \begin{subfigure}[b]{0.8\textwidth}
    \includegraphics[width=\textwidth]{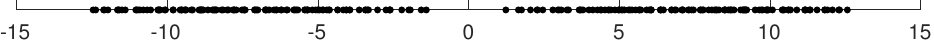}
         \caption{without preconditioner}
        \includegraphics[width=\textwidth]{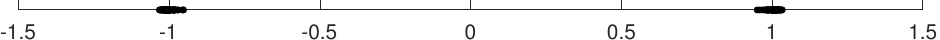}
         \caption{with $\cred{\widetilde{P}_n}$}
     \end{subfigure}
        \caption{Eigenvalues of $Y_n\cred{\widetilde{A}_n}$ with $\alpha_1=\alpha_2=1.5$ and $n=225$.}
        \label{fig:two_D_1_5_225}
\end{figure}

\begin{figure}[h!]
    \centering
     \begin{subfigure}[b]{0.8\textwidth}
    \includegraphics[width=\textwidth]{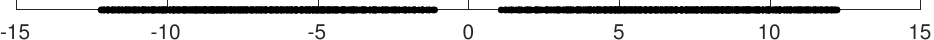}
         \caption{without preconditioner}
        \includegraphics[width=\textwidth]{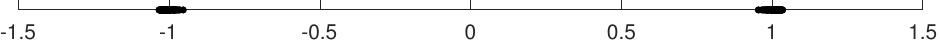}
         \caption{with $\cred{\widetilde{P}_n}$}
     \end{subfigure}
        \caption{Eigenvalues of $Y_n\cred{\widetilde{A}_n}$ with $\alpha_1=\alpha_2=1.5$ and $n=961$.}
        \label{fig:two_D_1_5_961}
\end{figure}

\begin{figure}[h!]
    \centering
     \begin{subfigure}[b]{0.8\textwidth}
    \includegraphics[width=\textwidth]{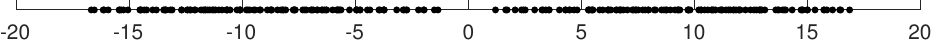}
         \caption{without preconditioner}
        \includegraphics[width=\textwidth]{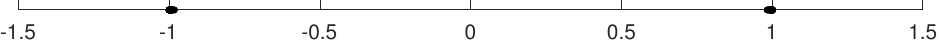}
         \caption{with $\cred{\widetilde{P}_n}$}
     \end{subfigure}
        \caption{Eigenvalues of $Y_n\cred{\widetilde{A}_n}$ with $\alpha_1=\alpha_2=1.9$ and $n=225$.}
        \label{fig:two_D_1_9_225}
\end{figure}

\begin{figure}[h!]
    \centering
     \begin{subfigure}[b]{0.8\textwidth}
    \includegraphics[width=\textwidth]{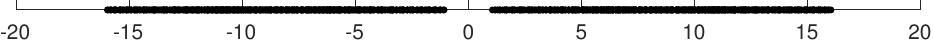}
         \caption{without preconditioner}
        \includegraphics[width=\textwidth]{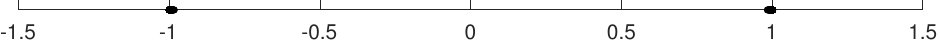}
         \caption{with $\cred{\widetilde{P}_n}$}
     \end{subfigure}
        \caption{Eigenvalues of $Y_n\cred{\widetilde{A}_n}$ with $\alpha_1=\alpha_2=1.9$ and $n=961$.}
        \label{fig:two_D_1_9_961}
\end{figure}

}
\end{example}

\cred{
\begin{example}\label{example:two_dim_second}
\rm{
In this example, we consider the two-dimensional fractional diffusion equation \eqref{eq:fde} with
$$
\begin{aligned}
d_{1,+} &=3, \quad d_{1,-}=1, \quad d_{2,+} = 2,  \quad d_{2,-} =1, \quad \Omega=(0,2) \times(0,2), \quad T=1,\\
{f\left(x_1, x_2, t\right)}&=  e^t x_1^2(2-x_1)^2 x_2^2(2-x_2)^2 \\
&-e^t x_2^2\left(2-x_2\right)^2 \sum_{i=2}^4 \frac{\binom{2}{i-2} 2^{4-i} i!\left[d_{1,+} x_1^{i-\alpha_1}+d_{1,-}\left(2-x_1\right)^{i-\alpha_1}\right]}{\Gamma\left(i+1-\alpha_1\right)(-1)^{i-2}} \\
& -e^t x_1^2\left(2-x_1\right)^2 \sum_{i=2}^4 \frac{\binom{2}{i} 2^{4-i} i!\left[d_{2,+} x_2^{i-\alpha_2}+d_{2,-}\left(2-x_2\right)^{i-\alpha_2}\right]}{\Gamma\left(i+1-\alpha_2\right)(-1)^{i-2}}.
\end{aligned}\\
$$
The exaction is given by $ u(x_1, x_2, t)=e^t x_1^2(2-x_1)^2 x_2^2(2-x_2)^2$. Note that the weighted and shifted Gr\"unwald scheme is used. The stated CPU times, iteration counts and error metrics apply only to the first time step. Additionally, we set $n_1=n_2$ and $\tau=T /\left( n_1+1\right)$. We did not implement MINRES-$MG(A_R)$ and MINRES-$MG(A_M)$, due to their high cost, as observed in the previous example. The GMRES solver \cite{Lin_etc_2023} was not applied in this example because it was developed specifically for the first-order shifted Gr\"unwald scheme. Since the exaction solution is available for this example, we define the error as $\textrm{Err}=\|\textbf{u}-\widetilde{\textbf{u}}\|_{\infty}$, where $\textbf{u}$ and $\widetilde{\textbf{u}}$ denote the exact solution and the approximate solution, respectively, resulted from the preconditioned MINRES solver.\\

Table \ref{Example_2_Table} presents the iteration count, CPU time, and error for the MINRES solver using the proposed preconditioner and without any preconditioner. These results are compared across different orders of fractional derivatives $(\alpha_1,\alpha_2)$ and $n$. Similar to the previous example, the findings summarized in Table \ref{Example_2_Table} suggest that MINRES-\( P_n \) achieves iteration counts that are independent of the mesh size. This establishes it as the most efficient method compared in virtually all cases, except when $(\alpha_1,\alpha_2)=(1.1,1.1)$, where the unpreconditioned solver MINRES-\( I_n \) is more efficient in terms of CPU times. Nonetheless, MINRES-\( P_n \) appears to be effective for all cases, confirming its stability and robustness to parameter variations.
}
\end{example}
}

\begin{table}[H]
\cred{
\small \caption{{Performance of MINRES-$I_n$ and MINRES-$P_n$ for Example 2 with $d_{1,+}=3,~d_{1,-}=1,~d_{2,+}=2$, and $d_{2,-}=1$.}}
\label{Example_2_Table}
\begin{center}
\begin{tabular}{cc|ccc|ccc}
\hline
\multirow{2}{*}{$(\alpha_1,\alpha_2)$} & \multirow{2}{*}{$n$} & \multicolumn{3}{c|}{MINRES-$I_n$} & \multicolumn{3}{c}{MINRES-$P_n$} \\ \cline{3-8}
                                  &                        & Iter         & CPU    & Err   & Iter         & CPU         & Err     \\ \hline
\multirow{4}{*}{(1.1,1.1)}      
                                  & 261121                 &  15          &  0.31     &5.3e-6      &  11            &  0.37            &5.3e-6  \\
                                & 1046529                 & 15           &  1.40    &1.3e-6        &   9           &  1.21         &1.3e-6     \\ 
                                & 4190209               & 13           &  7.28    &3.3e-7        &   9           &  9.26        &3.4e-7      \\ 
                                & 16769025 
                                                 &  13           & 25.72     & 8.2e-8   &  9            &  34.55   & 9.1e-8          \\\hline
\multirow{4}{*}{(1.1,1.5)}     
                                  
                                   & 261121                &   $>$100           &  -   &-          &  13            & 0.42           &1.8e-5    \\
                                  &1046529                &   $>$100           &  -     &-        &   11           &  1.41          &4.8e-6   \\ 
                                   & 4190209               &   $>$100           &  -     &-        &   11           &  10.90          &1.2e-6      \\ & 16769025                  & $>$100           &  -   &-         &11&41.27  &3.3e-7    \\\hline
\multirow{4}{*}{(1.1,1.9)}       
                                  
                                   & 261121               &  $>$100           &  -    &-      &11              & 0.34           &5.4e-6     \\
                                  & 1046529                &   $>$100           &  -     &-     &  11            & 1.42           &1.4e-6   \\ 
                                  & 4190209               &   $>$100           &  -     &-     &  11            & 10.65           &3.8e-7 \\ & 16769025                  &  $>$100           &  -     &-       &11 &41.11 &9.9e-8     \\\hline
\multirow{4}{*}{(1.5,1.1)}       
                                  
                                  & 261121                &   $>$100           &  -    &-        & 11             &  0.37          &2.2e-5    \\
                                  &1046529                &    $>$100           &  -    &-        &  11            & 1.41           &5.8e-6     \\
                                  & 4190209               &    $>$100           &  -    &-        &  11            & 10.85           &1.5e-6    \\ & 16769025                  &  $>$100           &  -     &-      &   11           &  40.92         &3.9e-7  \\\hline
\multirow{4}{*}{(1.5,1.5)}      
                                 
                                  & 261121                &  $>$100           &  -    &-       &    12          &   0.38        &2.1e-5     \\
                                  & 1046529                &   $>$100           &  -    &-       &   11           &  1.42          &5.7e-6      \\ 
                                   & 4190209               &   $>$100           &  -    &-       &   11           &  10.91          &1.5e-6    \\  & 16769025                  &     $>$100           &  -     &-       &  11            & 40.68           &3.9e-7  \\\hline
\multirow{4}{*}{(1.5,1.9)}       
                                  
                                  & 261121                &   $>$100           &  -   &-        &  13                      &0.42   
                                  &2.1e-5\\
                                  &1046529                &  $>$100           &  -    &-       &   13          &  1.62          &5.7e-6   \\ 
                                   & 4190209               &  $>$100           &  -    &-       &   12          &  12.27          &1.5e-6  \\  & 16769025                 &   $>$100           &  -   &-        &   11                    &40.88  
                                  &3.9e-7 \\\hline
\multirow{4}{*}{(1.9,1.1)}      
                                  
                                   & 261121                &   $>$100           &  -   &-      &     9         &   0.30          &6.2e-6   \\
                                  & 1046529                &     $>$100           &  -   &-         &    9          &  1.18          &1.6e-6     \\
                                   & 4190209               &     $>$100           &  -   &-         &    9          &  9.17          &4.3e-7  \\ & 16769025                  &   $>$100           &  -   &-       &    9          &  34.99         &1.1e-7     \\\hline
\multirow{4}{*}{(1.9,1.5)}      
                                
                                  & 261121                &    $>$100           &  -   &-      &        11      &    0.43         &1.8e-5  \\
                                  & 1046529                &         $>$100           &  -   &-        &       11       & 1.41          &4.8e-6 \\ 
                                   & 4190209               &         $>$100           &  -   &-        &       11       & 10.79         &1.2e-6   \\  & 16769025                &   $>$100           &  -    &-        &    11          &   40.73       &3.3e-7      \\ \hline
\multirow{4}{*}{(1.9,1.9)}       
                             
                                   & 261121  &   $>$100    &  -    &- &9 &  0.29 &6.2e-6   \\    
                                  & 1046529                &    $>$100           &  -   &-          &            9 & 1.18         &1.6e-6    \\
                                   & 4190209               &    $>$100           &  -   &-          &            9 & 9.13         &4.3e-7  \\ & 16769025                  &   $>$100           &  -   &-        &        9      &  34.64          &1.1e-7   \\ \hline                               
\end{tabular}
\end{center}
}
\end{table}

\section{Conclusions}\label{sec:conclusions}
We have developed a novel approach for solving space R.-L. fractional diffusion equations, utilizing a MINRES method based on multilevel Tau preconditioners. Our approach not only accounts for the ideal \cred{symmetric part} preconditioner $A_R$ pioneered in \cite{Pestana2019} but also offers improvements in both theoretical and numerical aspects. Our analysis suggests that employing the preconditioned MINRES method can lead to convergence that is independent of the mesh size. To validate the effectiveness of our proposed preconditioning strategy, we have provided numerical examples that demonstrate its superior capability. As future work, for a symmetrized multilevel Toeplitz system with \( Y_n T_n[f] \), we will develop a practical and effective preconditioner based on the absolute value function \( |f| \) and Tau matrices. It is expected to be more versatile and can be used not only for solving the space fractional diffusion equations currently under consideration, but also for a general symmetrized multilevel Toeplitz system. \cred{Moreover, we intend to investigate the performance of the proposed preconditioner in scenarios involving non-constant coefficient cases within a general domain under the GMRES framework. This presents a particularly challenging scenario as symmetrization appears infeasible under these settings. Consequently, our goal is to further develop preconditioners based on the symmetric part preconditioner $A_R$ are capable of achieving optimal convergence. Also, future work will explore this symmetric part preconditioning approach for other more challenging PDE problems, in addition to the R.-L. fractional diffusion equations currently under consideration.}

\section*{Acknowledgments}
The work of Sean Hon was supported in part by the Hong Kong RGC under grant 22300921 and a start-up grant from the Croucher Foundation.

\bibliographystyle{plain}

\end{document}